\documentclass{article}
\usepackage[utf8]{inputenc}
\usepackage[T1]{fontenc}
\usepackage{amsthm}
\usepackage{babel}
\usepackage{mathtools}
\usepackage{hyperref}
\usepackage{amsfonts}
\usepackage{amssymb}
\usepackage{geometry}
\usepackage{xcolor}
\usepackage{mparhack}

\hypersetup{
    colorlinks,
    citecolor=blue,
    linkcolor=red
}

\numberwithin{equation}{section}
\numberwithin{figure}{section}
\theoremstyle{plain}
\newtheorem{thm}{\protect\theoremname}[section]
\newtheorem*{thm*}{\protect\theoremname}
\theoremstyle{plain}
\newtheorem{lem}[thm]{\protect\lemmaname}
\newtheorem{cor}[thm]{\protect\corollaryname}
\newtheorem{proposition}[thm]{Proposition}
\newtheorem{claim}[thm]{Claim}
\newtheorem{definition}[thm]{Definition}

\newtheorem{theorem}{Theorem}

\providecommand{\corollaryname}{Corollary}
\providecommand{\lemmaname}{Lemma}
\providecommand{\theoremname}{Theorem}
\newcommand{\annotation}[1]{%
  \marginpar{\small\itshape\color{blue}#1}}

\global\long\def\defeq{\vcentcolon=}
\global\long\def\bool{\{0, 1\}}
\global\long\def\cube{\bool^n}

\newcommand{\remove}[1]{}

\global\long\def\E{\mathbb{E}}
\global\long\def\enc{G}
\global\long\def\encnoise{T_\rho}
\global\long\def\contrnoiseamount{\frac{1}{\sqrt{q}}}
\global\long\def\contrnoise{T_{\contrnoiseamount}}
\global\long\def\totnoise{T_{\rho \contrnoiseamount}}
\global\long\def\slice{\mathcal{S}_{1, n}}
\global\long\def\geslice{\mathcal{S}_{\geq 1, n}}
\DeclareMathOperator*{\EE}{\E}

\DeclareMathAlphabet{\mymathbb}{U}{BOONDOX-ds}{m}{n}
\newcommand{\NK}[1]{\textcolor{teal}{\{Nathan: #1\}}}
\newcommand{\NL}[1]{\textcolor{orange}{\{Noam: #1\}}}
\newcommand{\OM}[1]{\textcolor{cyan}{\{Omri: #1\}}}

\title{Sharp Hypercontractivity for Global Functions}

\author{Nathan Keller\thanks{Department of Mathematics, Bar-Ilan University. \texttt{Nathan.Keller@biu.ac.il}. Supported by the Israel Science Foundation (grant no.~2669/21).} , 
Noam Lifshitz\thanks{Einstein institute of Mathematics, Hebrew University. \texttt{noamlifshitz@gmail.com}. Supported by the European Research Council (StG no.~101163794) and by the Israel Science Foundation (grant no.~1980/22).} , and
Omri Marcus\thanks{Department of Mathematics, Bar-Ilan University. \texttt{omri.mar987@gmail.com}}}

\begin{document}

\maketitle

\begin{abstract}
    For a function $f \colon \{0,1\}^n \to \mathbb{R}$ with Fourier expansion $f=\sum_{S \subset \{1,2,\ldots,n\}} \hat f(S)\chi_S$, the hypercontractive inequality for the noise operator allows bounding norms of $T_\rho f = \sum_S \rho^{|S|} \hat f(S)\chi_S$ in terms of norms of $f$. If $f$ is Boolean-valued, the level-$d$ inequality allows bounding the norm of $f^{=d}=\sum_{|S|=d} \hat f(S)\chi_S$ in terms of $\mathbb{E}[f]$. These two inequalities play a central role in analysis of Boolean functions and its applications. 

    While both inequalities hold in a sharp form when the hypercube $\{0,1\}^n$ is endowed with the uniform measure, it is easy to show that they do not hold for more general discrete product spaces, and finding a `natural' generalization was a long-standing open problem. In [P. Keevash, N. Lifshitz, E. Long, and D. Minzer, Hypercontractivity for global functions and sharp thresholds, J. Amer. Math. Soc., 37:245-279, 2024], Keevash et al.~obtained a hypercontractive inequality for general discrete product spaces, that holds for functions which are `global' -- namely, are not significantly affected by a restriction of a small set of coordinates. This hypercontractive inequality is not sharp, which precludes applications to the symmetric group $S_n$ and to other settings where sharpness of the bound is crucial. Also, no sharp level-$d$ inequality for global functions over general discrete product spaces is known. 

    We obtain sharp versions of the hypercontractive inequality and of the level-$d$ inequality for global functions over discrete product spaces. Our inequalities open the way for diverse applications to
    extremal set theory, group theory, theoretical computer science, and number theory.
    We demonstrate this by proving quantitative bounds on the size of intersecting families of sets and vectors under weak symmetry conditions and by describing numerous applications that were obtained using our results. Those contain applications to the study of functions over the symmetric group $S_n$ -- including hypercontractivity and level-$d$ inequalities, character bounds, variants of Roth’s theorem and of Bogolyubov’s lemma, and diameter bounds, as well as an application to the Furstenberg-S\'{a}rk\"{o}zy problem on the maximal size of a subset of $\{1,2,\ldots,n\}$ which does not contain two elements that differ by a perfect square. 
\end{abstract}

\section{Introduction}


\subsection{Background}

Analysis of Boolean functions was initiated 
as a study of functions on the hypercube $\{0,1\}^n$ via their discrete Fourier expansion, that is, their unique representation as a multi-linear polynomial over the reals. The field has expanded greatly in the last 30 years, a rich theory was developed, and applications to a wide variety of fields -- e.g., percolation theory, group theory, hardness of approximation, machine learning, social choice, etc.~-- were established (see~\cite{Kalai-IMC,O'Donnell-IMC}). 

The basic tool, which lies behind most central results in the field, is the hypercontractive inequality for the noise operator, obtained by Bonami~\cite{Bonami}, Gross~\cite{gross}, and Beckner~\cite{Beckner}. For $0\leq \rho \leq 1$, the noise operator $T_{\rho} \colon L^2(\{0,1\}^n) \to L^2(\{0,1\}^n)$ is defined by $T_{\rho} f (x) = \E [f(y)]$, where $y$ is obtained from $x$ by leaving each coordinate unchanged with probability $\rho$ and replacing it by a random value with probability $1-\rho$. The basic form of the inequality asserts that when the hypercube is endowed with the uniform measure $\mu_{1/2}$, this operator is hypercontractive. 
\begin{theorem} \label{thm: classical hypercontractivity}
    Let $q\ge 2$ and let $\rho \le \frac{1}{\sqrt{q-1}}.$ 
    For any $f\in L^2(\cube,\mu_{1/2})$, we have
    \[
    \lVert T_\rho f\rVert_q\le \lVert f\rVert_2 .
    \]
\end{theorem}
An alternative approach to the noise operator is via the Fourier expansion. If $f=\sum_{S \subset \{1,2,\ldots,n\}} \hat f(S)\chi_S$, where $\chi_S=\prod_{i \in S} (2x_i-1)$, is the Fourier expansion of $f$, then $T_{\rho}f=\sum_S \rho^{|S|} \hat f(S)\chi_S$. Hence, it is natural to view the Fourier expansion of $f$ as divided into `levels', where the $d$'th level of $f$, defined as $f^{=d}=\sum_{|S|=d} \hat f(S) \chi_S$, is included in the eigenspace that corresponds to the eigenvalue $\rho^{|S|}$ of the noise operator. An extremely useful corollary of the hypercontractive inequality for Boolean-valued functions is the so-called level-$d$ inequality: 
\begin{thm} \label{thm: classical level-d}
    For any $f \colon (\{0,1\}^n,\mu_{1/2}) \to \{0,1\}$, and for any $d \le 2 \log (1 / \E[f])$, we have
    \[
    \lVert f^{=d}\rVert_2^2 \leq \E[f]^2 \left( 2e \frac{\log(1 / \E[f])}{d} \right)^d .
    \]
\end{thm}
Kahn, Kalai, and Linial~\cite{KKL} used a variant of Theorem~\ref{thm: classical hypercontractivity} for $\{-1,0,1\}$-valued functions to show that any function $f \colon \{0,1\}^n \to \{0,1\}$ whose expectation is bounded away from $0$ and $1$ has an influential variable, and Friedgut~\cite{Fri98} used it to show that if the sum of the influences of the variables on $f$ is small, then $f$ depends on only a few variables. 

\medskip This landscape breaks down as one moves from the uniform measure $\mu_{1/2}$ on $\{0,1\}^n$ to more general finite product spaces, and in particular, to the biased measure $\mu_p$ on $\{0,1\}^n$ (defined as $\mu_p(x)=p^{\sum x_i} (1-p)^{n-\sum x_i}$), which plays a central role in applications of analysis of Boolean functions to threshold phenomena (see, e.g.,~\cite{BK97,Fri99}). In the biased measure setting, the characters $\chi_S$ in the Fourier expansion have the form $\chi_S=\prod_{i \in S} \frac{x_i-p}{\sqrt{p(1-p)}}$. Hence, the function $f(x)=x_i$ has the Fourier expansion \[f(x)= p + \sqrt{p(1-p)}\frac{x_i - p}{\sqrt{p(1-p)}} = p+\sqrt{p(1-p)}\chi_{i}.\] Thus, $\lVert f\rVert_2=\sqrt{p}$ and $\lVert T_{\rho}f\rVert_q \geq \rho \cdot p^{1/q}$, which means that $\rho$ must be as small as $O(p^{1/2 - 1/q})$, for Theorem~\ref{thm: classical hypercontractivity} to hold.
Similarly, the function 
$g(x)=\prod_{i=1}^d x_i$ satisfies $\lVert g^{=d}\rVert_2^2=(p(1-p))^d \approx p^d=\mathbb{E}[g]$, and thus, a `correction factor' of order at least $\Omega(p \log(1/p))^{-d}$ is needed in order to fix Theorem~\ref{thm: classical level-d} (In both cases, these `corrections' are sufficient; see~\cite{K12,o2014analysis}).

As a result, the Kahn-Kalai-Linial theorem and Friedgut's theorem do not hold for general functions with respect to $\mu_p$ where $p=o(1)$, and numerous deep works obtained variants of them for restricted classes of functions, which exclude the counterexamples such as $f(x)=x_i$ (e.g.,~\cite{Fri99,Hat12}).

\subsection{The results of Keevash et al.~-- qualitative hypercontractivity for global functions}

Keevash, Lifshitz, Long, and Minzer~\cite{KLLM21} obtained a generalization of the hypercontractive inequality to the biased measure on the hypercube, under the additional assumption that the function $f$ is \emph{global}, i.e., is not significantly affected by a restriction of a small set of coordinates. Formally, for each $S \subset [n]=\{1,2,\ldots,n\}$, define the generalized influence $I_S(f)$ as
\[
I_S(f)=\E_{\mu_p}\left[\left(\sum_{x \in \{0,1\}^{|S|}} (-1)^{|S|-\sum x_i} f_{S \to x}\right)^2\right],
\]
where the restriction $f_{S \to x} \colon \{0,1\}^{[n] \setminus S} \to \mathbb{R}$ is defined by $f_{S \to x}(y)=f(x,y)$. The main result of~\cite{KLLM21} for $(\{0,1\}^n,\mu_p)$ reads as follows.
\begin{thm}\label{thm:KLLM-hypercontractivity}
    Let $q>2$ and let $\rho<(2q)^{-1.5}$. Let $f \colon  (\{0,1\}^n,\mu_p) \to \mathbb{R}$, and assume that $I_S(f) \leq \beta \lVert f\rVert_2^2$ for all $S \subset [n]$. Then $\lVert T_{\rho}f\rVert_q \leq \beta^{1/2-1/q}\lVert f\rVert_2$. 
\end{thm}

The novel technique of Keevash et al.~allowed them to prove a stronger variant of Bourgain's sharp threshold theorem~\cite[appendix]{Fri99}, to make progress on the inverse problem for the isoperimetric inequality on the Boolean cube~\cite{KK07}, and to obtain a $p$-biased analogue of the invariance principle of Mossel, O'Donnell and Oleszkiewicz~\cite{MOO10}. Even more recently, this technique was used in~\cite{ElKiLi22} to give a much shorter proof of the breakthrough result of Khot, Minzer and Safra on expansion of the Grassmann graph~\cite{KhMiSa23} which was the main mathematical ingredient in the proof of the 2-to-2 games conjecture in complexity theory.

A hypercontractive inequality under a somewhat different notion of globalness was obtained independently by O'Donnell and Zhao~\cite{OZ21}, who used it to prove a nearly-perfect expansion property for `pseudo-random' sets.

However, both Theorem~\ref{thm:KLLM-hypercontractivity} and the results of~\cite{OZ21} are not sharp in terms of the noise rate $\rho$, and this precludes applications to settings where a sharp version is crucial, such as various applications to the symmetric group $S_n$. Likewise, a sharp version of Theorem~\ref{thm: classical level-d} for the biased measure on the hypercube (without a `correction factor' that depends on $p$) is not known. 

\subsection{Our results}

In this paper we present sharp versions of the hypercontractive inequality and of the level-$d$ inequality for global functions over general finite product spaces. 
In the statements of our results, we use a \emph{restriction-based} notion of globalness, spelled out in the formulations of Theorems~\ref{cor:hypercontractivity for global functions - intro} and~\ref{thm:level d intro}, rather than a \emph{derivative-based} notion, like $I_S(f)$ used in Theorem~\ref{thm:KLLM-hypercontractivity}. The reason for this is that the restriction-based notion is more convenient for applications and is probably more natural from the combinatorial point of view. As the proofs require also using derivative-based notions of globalness, we establish transitions between the notions in Section~\ref{sec:hypercontractive}.

\begin{thm}\label{cor:hypercontractivity for global functions - intro}
     Let $q \ge 2$ and let $(\Omega,\mu)$ be a finite probability space. Let $f\colon (\Omega^n, \mu^n) \to \mathbb{R}$, and assume that \[\lVert f_{S \to x}\rVert_2 \leq r^{|S|} \lVert f\rVert_2\] for all $S \subset [n]$ and for all $x \in \Omega^S$. If $r \geq 1$ and $\rho \leq \frac{\log q}{32 r q}$,
    then 
    \[
    \lVert T_{\rho}f\rVert_q \le \lVert f\rVert_2.
    \]
    \end{thm}
We note that in the case of the biased measure on the hypercube, the globalness assumption in Theorem~\ref{cor:hypercontractivity for global functions - intro} is essentially equivalent, up to a constant factor, to the assumption in Theorem~\ref{thm:KLLM-hypercontractivity} (see Theorem~\ref{thm: strong hypercontractivity for global functions}).

The classical hypercontractive inequality (i.e., Theorem~\ref{thm: classical hypercontractivity}) implies that if $f \colon (\{0,1\}^n,\mu_{1/2}) \to \mathbb{R}$ is a linear function, then $\lVert f\rVert_q\le \sqrt{q}\lVert f\rVert_2.$ This corresponds to the behaviour of Gaussian random variables $X$, which satisfy $\lVert X\rVert_q = \Theta(\sqrt{q}\lVert X\rVert_2).$ When we switch from the uniform measure on the Boolean cube to the general product space setting, a new phenomenon arises, where functions behave like Poisson random variables, which satisfy $\lVert X\rVert_q=\Theta(\frac{q}{\log q}\lVert X\rVert_2).$ Theorem~\ref{cor:hypercontractivity for global functions - intro}, which demonstrates this phenomenon, is sharp for functions of general degree. This can be observed by inspecting the function $f \colon (\{0,1\}^n, \mu_{p}) \to \mathbb{R}$, defined by
\[f= \prod_{j=0}^{d-1}\left(\sum_{i=1}^{n/d} (x_{j(n/d)+i}-p) \right),\] 
where $p = \frac{d}{n}$ and $d$ divides $n$, as is shown in Proposition~\ref{lem: global hyper sharpness} below.

By relying on a slightly more involved version of Theorem \ref{cor:hypercontractivity for global functions - intro} (namely, Theorem~\ref{thm:hypercontractivity for global functions - alternative} and Corollary~\ref{cor:hypercontractivity for global functions - alternative}), we prove the following. 
\begin{thm}\label{thm:level d intro}
    There exists a constant $C>0$ such that the following holds. Let $(\Omega,\mu)$ be a finite probability space, and let $f\colon (\Omega^n, \mu^n) \to \{0,1\}$. Assume that for some $r>1$ and $d\le  \frac{1}{4}\log(1/\mathbb{E}[f])$, we have 
    $\mathbb{E}[f_{S\to x}] \le r^{|S|}  \mathbb{E}[f]$ for all sets $S$ of size $\le d$ and all $x\in \Omega^S.$ Then 
    \[ \lVert  f^{=d}\rVert_2^2 \leq \mathbb{E}^2[f] \left(\frac{Cr^2 \log(1/\mathbb{E}[f])}{d}\right)^d.
    \]
\end{thm}
For constant values of $r$, Theorem~\ref{thm:level d intro} is tight, up to the value of $C$. To see this, consider the function $f=\mathrm{AND}_t\colon \{0,1\}^n \to \{0,1\}$ defined by $f(x)=1$ if and only if $x_1=x_2=\cdots =x_t =1.$ It is is easy to see that for all $d\le t$ we have \[\lVert f^{=d}\rVert_2^2 \ge \mathbb{E}^2[f]\left(\frac{c\log(1/\mathbb{E}[f])}{d}\right)^d,\]
for an absolute constant $c$. This tightness example for the classical level-$d$ inequality (i.e., Theorem~\ref{thm: classical level-d}) can be lifted to a tightness example for Theorem~\ref{thm:level d intro} with respect to an arbitrary uniform product space $\left(\mathbb{Z}/m\mathbb{Z}\right)^n,$ with $m$ even, by applying $f$ to the input modulo 2. 

We remark that in order to obtain Theorem \ref{thm:level d intro}, it is crucial to use our Theorem~\ref{cor:hypercontractivity for global functions - intro} in place of Theorem~\ref{thm:KLLM-hypercontractivity} of Keevash et al. Indeed, if we have tried to apply Theorem~\ref{thm:KLLM-hypercontractivity}, our $\log(1/\mathbb{E}[f])$ factor would have been replaced by its cube $\log^3(1/\mathbb{E}[f])$. Saving on the power of $\log(1/\mathbb{E}[f])$ is crucial in the applications of our theorem: while Theorem~\ref{thm:level d intro} is applicable for sets of density as small as $2^{-\Omega(n)}$, a weaker variant with a $\log^3(1/\mathbb{E}[f])$ factor is applicable only when the density is at least $2^{-O(n^{1/3})}$.

\paragraph{Our techniques.} Unlike the proof of Theorem~\ref{thm:KLLM-hypercontractivity} which proceeds by replacing the coordinates of the function one-by-one by Gaussians and tracking the `cumulative error', our proof of Theorem~\ref{cor:hypercontractivity for global functions - intro} follows the proof strategy of the classical hypercontractive inequality -- proving the assertion for $n=1$ and then leveraging the result to general $n$ by tensorization. 

\paragraph{Applications.} The sharpness of Theorems~\ref{cor:hypercontractivity for global functions - intro} and~\ref{thm:level d intro} allows applying the `global hypercontractivity' technique in a wide range of settings where the qualitative results of Theorem~\ref{thm:KLLM-hypercontractivity} are insufficient. We demonstrate this in the two following sections, by presenting strong quantitative bounds on the size of `smeared' intersecting families of sets and vectors, and by describing a variety of results on functions on the symmetric group -- including hypercontractivity and level-$d$ inequalities, character bounds, variants of Roth's theorem and of Bogolyubov's lemma, and diameter bounds, as well as a new bound for the Furstenberg-S\'{a}rk\"{o}zy problem, that were obtained using our techniques. 

\subsection{Quantitative bounds on the size of smeared intersecting families}

A family $\mathcal{F}$ of sets is called \emph{intersecting} if for any $A,B \in \mathcal{F}$, $A \cap B \neq \emptyset$. The classical Erd\H{o}s-Ko-Rado theorem~\cite{EKR} asserts that for any $k<n/2$, the maximal size of an intersecting family of $k$-element subsets of $[n]$ is $\binom{n-1}{k-1}$. The Erd\H{o}s-Ko-Rado theorem is the cornerstone of a large field of research, which studies extremal problems on families of sets or other objects, under various conditions on their intersections (see the survey~\cite{FranklT16a}). 

A common feature of the Erd\H{o}s-Ko-Rado theorem and of many other central results in the field is that the extremal examples are obtained by `juntas' -- families $\mathcal{F}$ that depend on a few variables (which formally means that there exists a small $S \subset [n]$ and some $\mathcal{J} \subseteq 2^S$, such that $A \in \mathcal{F}$ if and only if $A \cap S \in \mathcal{J}$). This naturally gives rise to the question, whether a \emph{smeared} intersecting family -- namely, an intersecting family on which `many' variables have a `large' influence -- must be significantly smaller than the Erd\H{o}s-Ko-Rado maximum value.

In recent years, numerous \emph{qualitative} results in this spirit were obtained in various settings, asserting that if an intersecting family is `symmetric' in some sense (e.g., is transitive symmetric, or is regular, which means that any element of the ground set is included in the same number of sets of the family), then its size is much smaller than the maximal size of an intersecting family in the same setting~\cite{DFR08, EKNS19,EKN17,IhringerK19,keevash2021forbidden}. However, the bounds these results yield on the size of the family are far from optimal. The only effective bound was obtained recently by Kupavskii and Zakharov~\cite{kupavskii2022spread}, who showed that for any $k<cn/\log n$, the maximum size of a regular intersecting family of $k$-element subsets of $[n]$ is $\exp(-\Omega \left( \frac{n}{k} \right) ) \binom{n}{k}$ (which is almost tight, as was shown in~\cite{kupavskii2018regular}).
The result of Kupavskii and Zakharov~\cite{kupavskii2022spread} uses their \emph{spread approximation method} which builds on the techniques developed in the breakthrough result of Alweiss et al.~\cite{ALWZ21} on the Sunflower Lemma. This method was further developed in a series of very recent papers and was used to obtain numerous applications to intersection theorems in various settings (see, e.g.,~\cite{FranklK25,Kupavskii23partitions,Kupavskii23hereditary,KupavskiiN24}). 

\medskip Using our sharp hypercontractivity results, we obtain an effective bound on the measure (or size) of intersecting families, under the weaker assumption that the family is \emph{smeared} -- namely, that `many' variables have a non-negligible influence on it, even if it is far from being transitive symmetric or regular. Formally, we identify a family $\mathcal{F} \subset \{0,1\}^n$ with the Boolean-valued function $f=1_{\mathcal{F}}$. We denote $\delta(f)=\max_i |\hat f(\{i\})|$, and $m(f)=|\{i \colon \hat f(\{i\})^2 \geq \frac{\delta(f)^2}{2}\}|$, where the Fourier coefficients are taken w.r.t.~the measure $\mu_p$. We say that $f$ is a \emph{smeared level-$1$ function} (or, in short, that $f$ is \emph{smeared}) if $m(f) \geq 1/p^2$. We prove the following.
\begin{thm}
\label{theo:upper_bound_smeared_and_intersect-intro}
    There exist $c_1,c_2>0$ such that the following holds. Let $\mathcal{F}$ be an intersecting family, and assume that $1/\sqrt{m(1_{\mathcal{F}})} \leq p \leq 1/2$. Then
    $\mu_p(\mathcal{F}) \leq c_1 \exp \left( -c_2/p \right)$.
\end{thm}
As we show in Section~\ref{sec:intersecting}, for any $p>1/\sqrt{n}$ the assertion of the theorem is tight, up to a factor of $O(\log(1/p)\log n)$ in the exponent. Since a `regular' family $\mathcal{F}$ is clearly smeared (having all first-level Fourier coefficients of $1_{\mathcal{F}}$ equal), Theorem~\ref{theo:upper_bound_smeared_and_intersect-intro} implies the result of Kupavskii and Zakharov~\cite{kupavskii2022spread} for any $k>\sqrt{n}$, up to the constant factors $c_1,c_2$, while being applicable to a much wider range of functions. 

\medskip We also obtain the following quantitative result for \emph{vector-intersecting} families -- namely, for families $\mathcal{A} \subset [k]^n$ such that for any $x,y \in \mathcal{A}$ there exists $i$ such that $x_i=y_i$.
A family $\mathcal{A} \subset [k]^n$ is called \emph{transitive-symmetric} if it is invariant under the action of a transitive group of permutations $G \le S_n$.
\begin{thm}
\label{theo:vector_intersecting_upper_bound_intro}
    There exist $c_1,c_2>0$ such that the following holds. Let $n, k$ be natural numbers such that $2 \log n \leq k \leq \sqrt{n} \log n$. Let $\mathcal{A} \subseteq [k]^n$ be a transitive-symmetric vector-intersecting family. Then
    \[
        \frac{|\mathcal{A}|}{k^n} \leq c_1 \exp \left( -\frac{c_2 k}{\log n} \right) .
    \]
    For $k > \sqrt{n} \log n$, we have $\frac{|\mathcal{A}|}{k^n} \leq c_1 \exp \left( -c_2\sqrt{n} \right)$.
\end{thm}
As we show in Section~\ref{sec:intersecting}, for any $2 \log n \leq k \leq \sqrt{n} \log n$, the assertion of the theorem is tight, up to a factor of $O(\log k\log^2 n)$ in the exponent. Theorem~\ref{theo:vector_intersecting_upper_bound_intro} is the first quantitative upper bound in the vector-intersecting setting, after qualitative results were obtained in~\cite{DFR08,EKNS19,keevash2021forbidden}.

\subsection{Follow-up works}

Theorems~\ref{cor:hypercontractivity for global functions - intro} and~\ref{thm:level d intro} were used to obtain a variety of results concerning subsets of $S_n$ and functions on $S_n$, as well as a breakthrough result in number theory. We briefly describe some of these results in this subsection. Other very recent works that use our results alonside other techniques include a hardness theorem for $p$-pass streaming algorithms for the \emph{MaxCSP problem} obtained by Fei, Minzer and Wang~\cite{FeiMW25} and a solution of the Duke–Erd\H{o}s \emph{forbidden sunflower problem} in a significantly extended range of parameters obtained by Kupavskii and Noskov~\cite{KupavskiiN24}.

\subsubsection{Intersection problems for families of permutations}

A family $\mathcal{F}$ of permutations is called \emph{$t$-intersecting} if for any $\sigma,\tau \in \mathcal{F}$, there exist $i_1,\ldots,i_t$ such that $\sigma(i_j)=\tau(i_j)$, and is called $(t-1)$-avoiding if for any $\sigma,\tau \in \mathcal{F}$, $|\{i \colon \sigma(i)=\tau(i)\}| \neq t-1$.  
Deza and Frankl~\cite{DF77} conjectured that for all $n \geq n_0(t)$, the maximum size of a $t$-intersecting family is $(n-t)!$. Cameron~\cite{Cam86} conjectured that equality is obtained only for the families of the form $\mathcal{F}=\{\sigma \colon \sigma(i_1)=j_1,\ldots,\sigma(i_t)=j_t\}$, for some $i_1,\ldots,i_t,j_1,\ldots,j_t$. Ellis~\cite{Ellis14} conjectured that the same holds for $(t-1)$-avoiding families. 

In recent years, both problems attracted significant attention. For $t$-intersecting families, Ellis, Friedgut and Pilpel~\cite{EFP11}, Ellis and Lifshitz~\cite{EL22}, and Kupavskii and Zakharov~\cite{kupavskii2022spread} showed that the conjectures of Deza and Frankl and of Cameron hold for all $t$ up to  $O(\log \log n)$, $O(\log n/\log \log n)$, and $O(n/\log^2 n)$, respectively. For $(t-1)$-avoiding families, Ellis~\cite{Ellis14}, Ellis and Lifshitz~\cite{EL22}, and Kupavskii and Zakharov~\cite{kupavskii2022spread} proved the conjectures for $t$ up to $2$, $O(\log n/\log \log n)$ and $\tilde{O}(n^{1/3})$, respectively. Each of the works~\cite{EFP11,EL22,kupavskii2022spread} introduced entirely new tools into the study of these problems. 

In~\cite{KLMS23}, Keller, Lifshitz, Minzer and Sheinfeld used Theorem~\ref{thm:level d intro} to show that the conjectures of Deza and Frankl and of Cameron hold for all $t \leq O(n)$ in the $t$-intersecting setting and for all $t \leq O(\sqrt{n}/\log n)$ in the $(t-1)$-avoiding setting, thus obtaining a significant advancement on both problems. In particular, they proved the following.
\begin{thm}[{\cite[Theorem~1]{KLMS23}}]
    There exists $C>0$ such that the following holds for all $t \in \mathbb{N}$ and all $n>Ct$. Let $\mathcal{F} \subset S_n$ be a $t$-intersecting family of permutations. Then $|
    \mathcal{F}|\leq (n-t)!$. Furthermore, if $|\mathcal{F}|>0.75(n-t)!$ then $\mathcal{F}$ is contained in a family of the form $\mathcal{F'}=\{\sigma \colon \sigma(i_1)=j_1,\ldots,\sigma(i_t)=j_t\}$, for some $i_1,\ldots,i_t,j_1,\ldots,j_t$.
\end{thm}

\subsubsection{Hypercontractivity and level-$d$ inequality in the symmetric group} 
While the symmetric group $S_n$ is not a product space, it behaves like the product space $[n]^n$ (endowed with the uniform measure) in various aspects. In particular, Filmus, Kindler, Lifshitz, and Minzer~\cite{filmus2020hypercontractivity} found a way to deduce a hypercontractive inequality for functions over $S_n$ from a corresponding hypercontractive inequality for $[n]^n$. However, their reduction was lossy in the parameters. 

In~\cite{keevash2023sharp}, Keevash and Lifshitz found a way to make this reduction essentially lossless and obtained analogues of our main results (Theorems~\ref{cor:hypercontractivity for global functions - intro} and~\ref{thm:level d intro}) for functions over $S_n$, by reducing each result for $S_n$ to the corresponding statement that we proved for $[n]^n$.  
To present the results of~\cite{keevash2023sharp}, a few more definitions are needed. A function on $S_n$ is said to be of \emph{degree} $d$ if it can be written as a linear combination of products of $d$ \emph{dictators} $1_{\sigma(i)=j}$. Given $m$-tuples of distinct coordinates $I,J$ the function $f$ can be restricted to the set of permutations sending $I$ to $J$. The restriction is denoted by $f_{I\to J}$. 
The function $f$ is said to be $(r,\gamma)$-\emph{global} if for any $m$ and for all $m$-tuples $I,J$, we have $\lVert f_{I\to J}\rVert_2 \le r^{m}\gamma$. It is said to be $r$-global if it is $(r,\gamma)$-global for $\gamma=\lVert f\rVert_2$. The authors of~\cite{keevash2023sharp} proved the following analogue of Theorem~\ref{cor:hypercontractivity for global functions - intro}.
\begin{thm}[{\cite[Theorem~1.10]{keevash2023sharp}}]
\label{thm:hypercontractivity for global functions in the symmetric group}
There exists a family of self adjoint
operators $\mathrm{T}_{\rho}$ with the following properties.
\begin{enumerate}
\item (Large degree $d$ eigenvalues) If $f\colon S_{n}\to\mathbb{R}$ is a function of degree $\le d$ for $d \leq 10^{-5}n$, then 
\[
\left\langle \mathrm{T}_{\rho}f,f\right\rangle \ge\left(\frac{\rho}{72}\right)^{d}\lVert f\rVert_{2}^{2}.
\]
 \item Let $q\ge 2$. If $f$ is $\left(r,\gamma\right)$-global, and $\rho\le\frac{\log q}{32rq}$, then 
\[
\lVert \mathrm{T}_{\rho}f\rVert_{q}^{q}\le\gamma^{q-2}\lVert f\rVert_{2}^{2}.
\]
In particular, if $f$ is $r$-global, and $\rho\le\frac{\log q}{32rq}$, then 
\[
\lVert \mathrm{T}_{\rho}f\rVert_{q} \leq \lVert f\rVert_2.
\]
\item The operator $\mathrm{T}_{\rho}$ commutes with the action of $S_{n}$
from both sides. 
\end{enumerate}
\end{thm}
A function $f$ on $S_n$ is said to be $(r,\gamma_1,\gamma_2, d)$-\emph{biglobal} if for any $m \le d$ and for all $m$-tuples $I,J$ we have  $\lVert f_{I\to J}\rVert_2\le r^{m}\gamma_2$ and $\lVert  f_{I\to J} \rVert_1 \le r^{m}\gamma_1$. 
For a function $f\colon S_n\to \mathbb{R}$, we write $f^{=d}$ for the projection of $f$ onto the space of degree $\le d$ functions that are orthogonal to all degree $\le d-1$ functions.
Using Theorem \ref{thm:non bool lvl_d} (which is a slightly more general version of Theorem \ref{thm:level d intro}), the authors of~\cite{keevash2023sharp} proved the following strengthened analogue of Theorem~\ref{thm:level d intro}. 
\begin{thm}[{\cite[Theorems~4.1 and~1.8]{keevash2023sharp}}]\label{thm:level d on the symmetric group} 
For any $(r,\gamma_1,\gamma_2,d)$-biglobal function $f \in L^2(S_n)$, where $r>1$, $\gamma_2>\gamma_1>0$, and for any $d \leq \min(\frac{1}{4}\log(\tfrac{\gamma_2}{\gamma_1}),10^{-5}n)$, we have 
\[
\lVert f^{=d}\rVert_2^2 \leq \gamma_1^2 (10^6 r^2 d^{-1} \log(\tfrac{\gamma_2}{\gamma_1}))^d.
\]
In particular, there exists $C>0$ such that if $f \colon S_n \to \{0,1\}$ is $r$-global and $d \leq \min(\frac{1}{8}\log(\tfrac{1}{\mathbb{E}[f]}),10^{-5}n)$, then\footnote{Note that due to the difference between the $L_2$-based definition of globalness used in~\cite{keevash2023sharp} and the $L_1$-based globalness notion used in Theorem~\ref{thm:level d intro}, the term $r^4$ in Theorem~\ref{thm:level d on the symmetric group} corresponds to the term $r^2$ in Theorem~\ref{thm:level d intro}.}  
\[ \lVert  f^{=d}\rVert_2^2 \leq \mathbb{E}^2[f] \left(\frac{Cr^4 \log(1/\mathbb{E}[f])}{d}\right)^d.
\]
\end{thm}
\noindent This allowed accomplishing the following results.

\subsubsection{An analogue of Bogolyubov's lemma for $A_n$}

Bogolyubov's lemma for finite fields says (in the special case of binary fields) that if $A\subseteq \mathbb{F}_2^n$ has density $\ge \alpha,$ then $A+A+A+A$ contains an affine linear subspace of codimension $O(1/\alpha)$ (see, e.g.,~\cite{zhao2023graph}). It is one of the main ingredients in the classical proof of the Freiman-Ruzsa theorem. The main disadvantage of using Bogolyubov's lemma in the Abelian setting is the fact that the density of the subspace whose existence is guaranteed is exponentially smaller than the density of the original set. Had the dependency been polynomial, it would have implied the polynomial Freiman--Ruzsa conjecture. 
Keevash and Lifshitz \cite{keevash2023sharp} used Theorem \ref{thm:level d intro} to prove an analogue of Bogolyubov's lemma for the alternating group, with subspaces replaced by $m$-umvirates (that is, families of the form $\{\sigma \in S_n \colon \sigma(i_1)=j_1,\ldots,\sigma(i_m)=j_m\}$, for some fixed $i_1,\ldots,i_m,j_1,\ldots,j_m$). In their setting, the measure of the $m$-umvirates is polynomial in the density, provided that the density is at least $e^{-cn^{0.99}}$. 
We say that a set $A$ is \emph{symmetric} if $A=A^{-1}.$ We write $AB = \{ab \colon \,a\in A,B\in B\}$, and $A^i$ for the $i$-fold product of $A$ with itself. 
\begin{thm}[{\cite[Theorem~1.3]{keevash2023sharp}}]
There exists $c>0$ so that for any integer $M<\log n$ and any symmetric set $A \subset A_n$ having density $\mu(A) \geq e^{-cn^{1-1/M}}$, the set $A^{8M}$ contains a subgroup $U_I$ with $\mu(U_I) \geq \mu(A)^{5M}$.
\end{thm}

\subsubsection{An analogue of Roth's theorem for $S_n$}

Keevash and Lifshitz \cite{keevash2023sharp} also considered the problem of determining the largest set $A\subseteq S_n$ that does not contain a solution to the equation $xz=y^2$ -- that is, the problem of finding the strongest possible analogue for $S_n$ of the classical theorem of Roth~\cite{Roth53} on arithmetic progressions.  

For general groups, a qualitative analogue of Roth’s theorem (as well as a similar result for more general equations) was established by Kr\'{a}l', Serra and Vena~\cite{KralSV09}. Their ‘regularity type’ bounds for the maximum density were improved by Sanders~\cite{Sanders17} to doubly logarithmic bounds in the specific case of the equation $xz=y^2$. Keevash and Lifshitz used Theorem \ref{thm:level d on the symmetric group} to obtain 
an improvement of the density upper bound in $S_n$ from $(\log n)^{-c}$ to $n^{-c\log n}$. 
\begin{thm}[{\cite[Theorem~1.5]{keevash2023sharp}}]
    There is an absolute constant $c>0$ so that any $A \subset S_n$ with no solution for the equation $xz=y^2$, has $\mu(A) \leq n^{-c\log n}$.
\end{thm}



\subsubsection{Diameter bounds for $S_n$}

The distance between two vertices in a graph is the number of edges in the shortest path that connects them. The \emph{diameter} of a graph is the largest distance between two vertices. 
For a group $G$ and $A \subset G$, the \emph{covering number} of $A$ is the minimal $\ell$ for which $A^{\ell}=G.$ The diameter of a Cayley graph $\mathrm{Cay}(G,A)$ is equal to the covering number of $A$.

The problem of determining covering numbers of sets $A\subseteq G$ is well studied for finite simple groups. A striking result of Liebeck and Shalev \cite{liebeck2001diameters} solves this problem completely when $G$ is a finite simple group and $A$ is a conjugacy class. They showed that the covering number of such conjugacy classes is always $\Theta\left(\log_{|A|}{|G|}\right).$ 

For the groups $S_n$ and $A_n$, a well-known longstanding folklore conjecture asserts that the diameter of any connected Cayley graph on $A_n$ or on $S_n$ is polynomial in $n$ (see~\cite{BS92}). A celebrated result of Helfgott and  Seress~\cite{helfgott2014diameter} shows a quasi-polynomial upper bound on the diameter, while Nikolov and Pyber~\cite{nikolov2011product} showed that Gowers'~\cite{gowers2008quasirandom} result about product mixing implies that any set in $A_n$ of density $\ge (n-1)^{-1/3}$ has covering number 3. 

Following Nikolov and Pyber~\cite{nikolov2011product}, Keevash and Lifshitz \cite{keevash2023sharp} considered the following more general problem: What is the largest possible diameter of a connected Cayley graph $\mathrm{Cay}(A_n, A)$ for a set $A$ of density $\alpha$? One example to keep in mind is the union of the $t$-umvirates of all permutations fixing $\{1,\ldots ,t\}$ with the set of transpositions $\{(1,2),(2,3),\ldots, (t,t+1)\},$ whose covering number is $\Theta(t^2).$ The authors of~\cite{keevash2023sharp} showed that for a fixed $\epsilon$ and $\alpha \ge e^{-n^{1-\epsilon}}$, the above construction is sharp up to a constant. Specifically, they showed the following.
\begin{thm}[{\cite[Theorem~1.1]{keevash2023sharp}}]
    There exist absolute constants $C,c > 0$ such that if $A \subset A_n$ is symmetric and satisfies $\mu(A) \geq e^{-cn^{1-\epsilon}}$ where $\epsilon>\frac{1}{\log n}$, then the diameter of the Cayley graph $\mathrm{Cay}(A_n, A)$ is at most $8\epsilon^{-1}+C\epsilon^{-2}(\log_n(1/\mu(A))^2$.
\end{thm}
This essentially solves this problem completely for sets $A$ of density $\ge e^{-cn}.$   

\subsubsection{Character bounds in the symmetric group}
 Lifshitz and Marmor~\cite{Lifshitz2023bounds} used Theorem \ref{thm:hypercontractivity for global functions in the symmetric group} to obtain the following bound on the norms of the characters.
\begin{thm}[{\cite[Theorem~1.12]{Lifshitz2023bounds}}]
There exists an absolute constant $C>0$, such that the following holds. Let $d$ be a positive integer and let $q\ge 2$.  Let $\lambda$ be a partition of $n$ and suppose that $\lambda_1=n-d$.  Then 
$\lVert \chi_\lambda\rVert_q \le \left(\frac{Cq}{\log q}\right)^d\left(\frac{\chi_{\lambda}(1)d^d}{n^d}\right)^{1-2/q}.$ 
\end{thm}
They also showed that their bound is sharp, up to the value of the constant $C$, unless $d$ is tremendously large with respect to $q$ (see~\cite[Theorem~1.13]{Lifshitz2023bounds}). This allowed the authors of~\cite{Lifshitz2023bounds} to show that there exits an absolute constant $c>0$, such that for every conjugacy class $A$ of $S_n$ of density $\frac{|A|}{n!} \geq e^{-c\sqrt{n}\log n}$, the mixing time of the Cayley graph $\mathrm{Cay}(A_n,A)$ is at most two. The lower bound $e^{-c\sqrt{n}\log n}$ on the density of $A$ is optimal, up to the value of the implicit constant $c$. This allowed them to improve various results of Larsen and Shalev \cite{larsen2008characters} on character bounds, diameter bounds, and mixing of conjugacy classes. 

\subsubsection{New upper bound for the Furstenberg-S\'{a}rk\"{o}zy problem}

The well-known Furstenberg-S\'{a}rk\"{o}zy problem asks, what is the maximal size $s(n)$ of a set $S \subset [n]$ that does not contain two elements whose difference is a perfect square. Lov\'{a}sz conjectured that $s(n)=o(n)$. This conjecture was proved independently by Furstenberg~\cite{Furstenberg77} and by S\'{a}rk\"{o}zy~\cite{Sarkozy78a,Sarkozy78b}. The proof of Furstenberg uses ergodic-theoretic methods, and thus, does not provide any quantitative bound on $s(n)$. S\'{a}rk\"{o}zy's proof uses Fourier-analytic methods and provides the bound $s(n) \ll n \frac{(\log \log n)^{2/3}}{(\log n)^{1/3}}$. Pintz, Steiger and Szemer\'{e}di~\cite{PSS88}, and consequently, Bloom and Maynard~\cite{BM22} improved the bound upper bound to $n (\log n)^{-\frac{1}{12}\log \log \log \log n}$, and to $n (\log n)^{-c\log \log \log n}$, respectively. The best known lower bound is $n^{0.733412}$, proved by Lewko~\cite{Lewko15}. The question whether $s(n)<n^{1-c}$ for some $c>0$ is widely open.

In a recent breakthrough, Green and Sawhney~\cite{GS24} used our results to obtain a hugely improved upper bound of $s(n) \leq n \exp(-c(\log n)^{1/2})$, for an absolute constant $c$. The main technical tool used by Green and Sawhney is the following variant of Theorem~\ref{thm:level d intro}:
\begin{thm}[{\cite[Theorem~1.2]{GS24}}]
\label{thm:GS}
    Set $C_0:= 2^{13}$. Let $\alpha \in (0,1/2)$ and $n \geq 1$ be parameters with $\alpha>2n^{-1/2}$. Let $\mathcal{Q}$ be a set of pairwise coprime positive integers such that $\max_{q \in \mathcal{Q}} q \leq n^{1/(32\log(1/\alpha))}$.
    Let $1 \leq d \leq 2^{-7}\log(1/\alpha)$. Let $f \colon  [n] \to \mathbb{C}$ be such that $|f(x)| \leq 1$ for all $x$. Then either
    \[
    \sum_{\substack{S \subset \mathcal{Q} \\ |S|=d}} \sum_{\substack{a \bmod (\prod_{q \in S}q) \\ q \in S \Rightarrow q \nmid a}} \left| \hat f \left(\frac{a}{\prod_{q \in S}q}\right) \right|^2 \leq \alpha^2 n^2 \left(\frac{C_0 \log(1/\alpha)}{d}\right)^d,
    \]
    or else for some set $S \subset \mathcal{Q}$, $1 \leq |S| \leq 2\log(1/\alpha)$, and for some $r \in \mathbb{Z}$, the average of $|f(x)|$ on the progression $P=\{x \in [n] \colon x \equiv r \bmod(\prod_{q \in S} q) \}$ is greater than $2^{|S|}\alpha$. 
\end{thm}
Essentially, Theorem~\ref{thm:GS} asserts that either $f$ satisfies a level $d$ inequality, or it is not global, in the sense that its average is increased significantly when restricted to some `not-too-sparse' arithmetic progression. The proof uses Corollary~\ref{cor:hypercontractivity for global functions - alternative}, as well as techniques from Section~\ref{sec:level-d}.

The method of~\cite{GS24} allows applying our results to functions over $[n]$, by introducing a product space structure using a `basis' $\mathcal{Q}$ of pairwise coprime small factors. Thus, it opens the way for further applications of our results to number theory.

\subsection{Organization of the paper}

In Section~\ref{sec: background} we present basic notions and lemmas that will be used throughout the paper. In Section~\ref{sec:Gaussian} we show that norms of the noise operator can be bounded using the norms of certain `Gaussian encodings' of the function. The proofs of the sharp hypercontractive inequality for global functions (namely, Theorem~\ref{cor:hypercontractivity for global functions - intro}) and of the sharp level-$d$ inequality (namely, Theorem~\ref{thm:level d intro}) are presented in Sections~\ref{sec:hypercontractive} and~\ref{sec:level-d}, respectively. Finally, we apply our results to bound the size of smeared intersecting families of sets and vectors in Section~\ref{sec:intersecting}.
\section{Background} \label{sec: background}

\subsection{Functions over the hypercube}

As was mentioned above, the classical results in discrete Fourier analysis were obtained for functions $f \colon \{0,1\}^n \to \mathbb{R}$, where the hypercube $\{0,1\}^n$ is endowed with the product measure $\mu_p(x)=p^{|\{i \colon x_i=1\}|}(1-p)^{|\{i \colon x_i=0\}|}$. While in this paper we consider functions over general finite product spaces, we begin with presenting the hypercube case, and then develop the notions for the general case as generalizations of the `hypercube case' notions.

\paragraph{Fourier expansion.} Denote $[n]=\{1,2,\ldots,n\}$. For each $i \in [n]$, let $\chi_i(x)=
\frac{x_i-p}{\sqrt{p(1-p)}}$, and for each $S \subset [n]$, let $\chi_S=\prod_{i\in S} \chi_i$. It can be easily verified that the set of characters $\{\chi_S\}_{S \subset [n]}$ is an orthonormal basis for the space of functions $\mathcal{U}_n = \{f \colon \{0,1\}^n \to \mathbb{R}\}$ with respect to the standard inner product $\langle f,g \rangle = \mathbb{E} [fg]$. Hence, each $f \in \mathcal{U}_n$ can be uniquely represented as $f=\sum_{S \subset [n]} \hat f(S)\chi_S$, where $\hat f(S)=\langle f,\chi_S \rangle$. This expansion is called the Fourier expansion of $x$. 

\paragraph{Restrictions, derivatives, and influences.} For $f \in \mathcal{U}_n$ and $S \subset [n]$, the \emph{restriction} $f_{S \to x} \colon \{0,1\}^{[n] \setminus S} \to \mathbb{R}$ is defined as $f_{S \to x}(y)=f(x,y)$ for all $y \in \{0,1\}^{[n] \setminus S}$. The \emph{discrete derivative} in the $i$'th direction is defined as $D_i(f)=\sqrt{p(1-p)} \cdot (f_{i \to 1}-f_{i \to 0})$. The \emph{influence} of the $i$'th coordinate on $f$ is $I_i(f)=||f_{i \to 1}-f_{i \to 0}||_2^2$. The discrete derivative can be applied sequentially, to obtain 
\[D_S(f) =\sqrt{p(1-p)}^{|S|}\sum_{x \in \{0,1\}^S} (-1)^{|S|-|\{i \colon x_i=1\}|}f_{S \to x}.\]
Discrete derivatives (and subsequently, influences) have a convenient representation in terms of the Fourier expansion:
\[D_S(f)=\sum_{T:S \subset T} \hat f(T)\chi_{T \setminus S}, \qquad \mbox{and} \qquad I_i(f)=\frac{1}{p(1-p)}\sum_{S:i \in S} \hat f(S)^2.\]

\paragraph{The noise operator.} For $0<\rho<1$, let $N_{\rho}(x)$ be obtained from $x$ by leaving each coordinate $x_i$ of $x$ unchanged with probability $\rho$ and replacing it by a random $y_i \sim \mu_p(\{0,1\})$ with probability $1-\rho$. The noise operator $T_{\rho} \colon \mathcal{U}_n \to \mathcal{U}_n$ is defined by $T_{\rho} f(x) = \mathbb{E}_{y \sim N_{\rho}(x)} f(y)$. It can be easily verified that the eigenvectors of the noise operator are precisely the characters $\{\chi_S\}$, and that for any $f = \sum_S \hat f(S) \chi_S$, we have $T_{\rho} f = \sum_S \rho^{|S|} \hat f(S) \chi_S$.   

For more details on the above classical definitions and their properties, the reader is referred to~\cite{o2014analysis}.

\subsection{Functions over general finite product spaces}

Let $(\Omega, \mu)$ be a finite probability space, and let $(\Omega^n, \mu^n)$ be the product space, where $\mu^n$ is the product probability measure. We will mostly omit the superscript $n$ in the expression $\mu^n$, when the dimension is clear from context.
For $S \subseteq [n]$ we denote by $\Omega^S$ the projection of $\Omega^n$ that leaves only the coordinates in $S$, and $\mu^S$ will be the corresponding projected measure. We denote by $L^2(\Omega, \mu)$ the space of real-valued functions on $\Omega$. 

In the case of the hypercube, for $f=\sum_S \hat f(S) \chi_S$, the term $\hat f(S)\chi_S$ represents the `part of $f$ which depends precisely on $S$'. The Efron-Stein decomposition generalizes this property to functions over general spaces. It will be convenient for us to introduce the Efron-Stein decomposition using the two following operators.

\paragraph{Laplacians and averages.} In the one-dimensional case, the \emph{Laplacian} operator $L$ is defined using the expectation operator $E$: For any $g \colon \Omega \to \mathbb{R}$,
\[E g = \E[g], \quad \mbox{and} \quad L g = g - E g .\]
For functions $\{f_i\colon \Omega \to \mathbb {R}\}_{i \in [n]}$, we write $f_1 \otimes \cdots \otimes f_n$ for the function $f \colon \Omega^n \to \mathbb{R}$ defined by $x\mapsto \prod_{i=1}^n f_i(x_i).$ For linear operators $T_i\colon L^2(\Omega, \mu)\to L^2(\Omega, \mu)$ we define $T_1\otimes \cdots \otimes T_n \colon L^2(\Omega^n, \mu) \to L^2(\Omega^n, \mu)$ by setting $f_1\otimes \cdots \otimes f_n\mapsto T_1 f_1 \otimes \cdots \otimes T_n f_n$ and extending linearly.   
    
We can now generalize the expectation operator and the Laplacian to the $n$-dimensional case by taking tensors with the identity operator. For each $i \in [n]$, we set
    \[
        E_i f = (I \otimes \cdots \otimes \underbrace{E}_i \otimes \cdots \otimes I) f, \quad L_i f = (I \otimes \cdots \otimes \underbrace{L}_i \otimes \cdots \otimes I) f.
    \]
More generally, for $S \subset [n]$,
    \[
        E_S[f](x) = \bigotimes_{i\in S} E_i\otimes \bigotimes_{i\notin S} \text{id}, \quad 
        L_S f = \bigotimes_{i\in S} L_i \otimes \bigotimes_{i\in S^c} \text{id}.
    \]
Combinatorially, the operator $E_S$ is the expected value obtained when resampling the coordinates in $S$ according to $\mu$. The Laplacians $L_S$ have a combinatorial description in terms of the averaging operators $E_S$: For any $S\subset [n]$, we have $L_S = \sum_{T \subseteq S} (-1)^{|T|} E_T$.

\paragraph{The Efron-Stein decomposition.} In the one-dimensional case, the space $L^2(\Omega, \mu)$ can be partitioned into two parts: $V_{=\varnothing}$, which consists of constant functions, and $V_{=\{1\}}$, which consists of functions with expectation $0$. The projection on $V_{=0}$ is the averaging operator $E$ and the projection on $V_{=\{1\}}$ is the Laplacian $L=f-\mathbb{E}[f].$ 
    
To generalize this partition to the $n$-dimensional case, we may use tensorization to obtain the operators 
    \[
    \{E_{T^c} L_T\}_{T\subseteq [n]} =\bigotimes_{i\in T} L_i \otimes \bigotimes_{i\in T^c} E_i.
    \]
Using these operations, we may obtain an orthogonal decomposition of the space $L^2(\Omega^n,\mu)$ into the spaces $V_{=T} = \{ E_{T^c} L_T [f] \colon f\in L^2(\Omega^n, \mu)\}$. Note that the space $V_{=T}$ consists of functions that depend only on the coordinates in $T$, and thus, serves as a generalization of the space $\mathrm{span}(\chi_T)$ from the `hypercube' case, in which we have  
$E_{T^c} L_T [f] = \hat{f}(T) \chi_T$. 

The \emph{Efron-Stein decomposition} of $f \in L^2(\Omega^n,\mu)$ is 
\[
    f = \sum_{T \subset [n]} f^{=T}, \qquad \mbox{where} \qquad f^{=T}=E_{T^c}L_T[f].
\] 
We write $f^{=d}=\sum_{|S|=d}f^{=S}$.
     
Note that unlike the hypercube case, the spaces $V_{=T}$ are no longer 1-dimensional. However, they still convey the notion of `dependence on exactly the variables in $T$', which turns out to be sufficient for our purposes.

The following lemma summarizes basic properties of the Efron-Stein decomposition. The simple proof is omitted.
\begin{lem} \label{lem:laplacian_efron_stein}
Let $V = L^2(\Omega^n, \mu)$. 
\begin{enumerate}
    \item The spaces $V^{=T}$ are an orthogonal decomposition of $L^2(\Omega^n, \mu)$. Therefore, for every $f,g \in V$ we have $\langle f, g\rangle = \sum_{T\subseteq [n]} \langle f^{=T}, g^{=T} \rangle$.

    \item For any $f\in V$, we have
        \[
            L_S f = \sum_{T \supseteq S} f^{=T}, \qquad \mbox{and} \qquad E_S f = \sum_{T \subseteq S^c} f^{=T}.
        \]
\end{enumerate}
\end{lem}
For more information on the Efron-Stein decomposition, we refer the reader to~\cite[Chapter~8]{o2014analysis}.

\paragraph{Restrictions and derivatives.}
Like in the hypercube case, for $f \in V$, $S \subset [n]$, and $x \in \Omega^S$, the restriction $f_{S \to x} \colon \Omega^{[n] \setminus S} \to \mathbb{R}$ is defined as $f_{S \to x}(y)=f(x,y)$. As was suggested in~\cite{keevash2021forbidden}, the discrete derivatives can be generalized using restrictions of the Laplacians. For all $S \subseteq [n]$, $x \in \Omega^n$, define
    \[D_{S, x} f = [L_S f]_{S \rightarrow x_S}.\]
We use the same notation when $x\in \Omega^T$ for an arbitrary set $T$ containing $S$, and write $D_{i,x}$ instead of $D_{\{i\},x}$ for brevity. 

\paragraph{The noise operator.} The noise operator, defined above in the `hypercube' setting, can be naturally defined over a general product space $(\Omega_1 \times \cdots \times \Omega_n,\mu_1 \times \cdots \times \mu_n)$ as $T_{\rho} f(x) = \mathbb{E}_{y \sim N_{\rho}(x)} f(y)$, where $N_{\rho}(x)$ is obtained from $x$ by leaving each coordinate $x_i$ of $x$ unchanged with probability $\rho$ and replacing it by a random $y_i \sim \mu_i$ with probability $1-\rho$.
It can be easily verified that in the space $(\Omega^n,\mu)$ (i.e.,~where $(\Omega_i,\mu_i)=(\Omega,\mu)$ for all $i$), like in the hypercube case, the eigenspaces of $T_\rho$ are exactly the spaces $V^{=T}$, and that for any $f \in V$, we have
\begin{equation}\label{eq:noise_efron_stein}
T_\rho f = \sum_{S \subseteq [n]} \rho^{|S|} f^{=S}. 
\end{equation}
In particular, it is clear that $T_{\rho}$ is multiplicative: we have $T_{\rho_1} \circ T_{\rho_2} = T_{\rho_1 \cdot \rho_2}$.

The following relations between the Laplacian, the derivatives, and the noise operator, can be easily verified.
\begin{lem} \label{lem:swap laplace noise derivative}
    Let $f \in V$. Then for any $\rho>0$, any $S \subset [n]$, and any $x \in \Omega^S$, we have
    \[
    L_S T_{\rho} f = T_{\rho} L_S f, \qquad \mbox{and} \qquad D_{S,x} T_\rho f = \rho^{|S|} T_{\rho} D_{S,x} f.
    \]
\end{lem}

\section{Bounding Norms Using Gaussian Encodings}
\label{sec:Gaussian}
In this section we show that the $q$-norm of the noise operator, namely, $\lVert T_\rho f\rVert_q$, can be bounded using the norms of certain `Gaussian encodings' of $f$. This will allow us to reduce the hypercontractive inequalities we want to prove to the Gaussian setting in which similar hypercontractive inequalities are known.  

\subsection{Gaussian encodings}

Let $(\Omega, \mu)$ be a probability space with $k$ elements, and let $B=\{ 1, f_1, \ldots, f_{k-1} \}$ be an orthonormal basis for $L^2(\Omega,\mu)$.
Let $\gamma$ be the Gaussian probability measure on $\mathbb{R}^{k-1}$, where the coordinates are independent normal (i.e.,  $N(\mathbf{0}, 1)$) Gaussian random variables. We define an operator $\enc \colon L^2(\Omega, \mu) \to L^2(\mathbb{R}^{k-1}, \gamma)$ by setting
\[
\enc \Big(\alpha_0+\sum_{i=1}^{k-1} \alpha_i f_i \Big) = \alpha_0+\sum_{i=1}^{k-1} \alpha_i z_i
\]
for all $\alpha_0,\ldots,\alpha_{k-1} \in \mathbb{R}$, where $z_i$ is the projection on the $i$'th coordinate. By tensorization, for any $S \subset [n]$ we obtain an operator
\[\enc_S\colon L^2(\Omega^n, \mu) \to L^2(\Omega^{S^c} \times \left( \mathbb{R}^{k-1}\right)^S, \mu^{ S^c} \otimes \gamma^{ S} ),\]
defined by
\[ \enc_S = \bigotimes_{i \in S} \enc \otimes \bigotimes_{i \in S^c} \text{id}.\]
For the sake of convenience, we shorten $G_{\{i\}}$ to $G_i$.

\subsection{Bounding norms using the encodings}

We aim at showing that $\lVert T_{\rho}f\rVert_q$ can be bounded using norms of the encodings $\enc_S(f)$. Specifically, we prove the following:
\begin{proposition}
    \label{lem:tenzorization}
    Let $\rho\le \frac{1}{3}$, let $q \ge 2$, and let $f \in L^2(\Omega^n, \mu)$. Then
        \[ \lVert  T_\rho f \rVert_q^q \leq \sum_{S \subseteq [n]} \beta^{q |S|} \lVert  (L_S \circ \enc_{S^c}) f \rVert_q^q,\]
        where $\beta=\rho \left(1 + \frac{2 (q-2)}{\log(1/\rho)}\right)$.
\end{proposition}
Note that here and throughout the paper, `$\log$' means the natural logarithm. 

\medskip We prove the proposition in three stages. First, we prove it in the case $(k=2,n=1)$, where each function $f$ has the form $f=\alpha_0+\alpha_1 f_1$ and we have $T_\rho f=\alpha_0 + \rho \alpha_1 f_1$. Then, we prove it for general linear functions, i.e., for $f=\alpha_0+\sum_{i=1}^{k-1} \alpha_i f_i$, where $T_{\rho} f = \alpha_0 + \rho \sum_{i=1}^{k-1} \alpha_i f_i$. Finally, we deduce  Proposition~\ref{lem:tenzorization} from the linear case using tensorization. 

\subsubsection{Linear functions for $k=2$}

In this case, we may write $f=\alpha_0+\alpha_1 f_1$ and the assertion of Proposition~\ref{lem:tenzorization} reads 
\begin{equation}\label{Eq:k=2,n=1}
    \lVert  T_\rho f \rVert_q^q \leq \lVert \enc(f)\rVert_q^q + \beta^q \lVert \alpha_1 f_1\rVert_q^q, 
\end{equation}
since $L_{\{1\}}(\alpha_0+\alpha_1 f_1)=\alpha_1 f_1$. We prove the following slightly more general lemma:
\begin{lem}
    \label{lem:n=1_one_var_general}
    Let $0 < \rho \leq 1/3$ and $q \geq 2$ be real numbers.
    Let $X$ be a real random variable with $\E[X] = 0$ and $\E[X^2] = 1$, and let $Z$ be a normal Gaussian random variable.
    Then for any $d \in \mathbb{R}$ we have
    \[ \lVert 1 + \rho d X \rVert_{q}^{q} \le \lVert 1 + d Z\rVert_{q}^{q} + \beta^{q} \lVert d X\rVert_q^q,\]
where $\beta=\rho \left(1 + \frac{2 (q-2)}{\log(1/\rho)}\right)$.
\end{lem}
Lemma~\ref{lem:n=1_one_var_general} clearly implies~\eqref{Eq:k=2,n=1}. Indeed, for $\alpha_0=0$ (i.e., $f=\alpha_1 f_1$), we have $T_{\rho}f=\rho f$ and hence~\eqref{Eq:k=2,n=1} follows from the inequality $\lVert \rho f \rVert_q^q \leq \beta^q \lVert f\rVert_q^q$ which holds trivially. For $\alpha_0\neq 0$, by dividing both sides of~\eqref{Eq:k=2,n=1} by $|\alpha_0|^q$ we obtain a special case of Lemma~\ref{lem:n=1_one_var_general}.
    
\begin{proof}
    We may assume $q > 2$, since it is easy to check that for $q=2$ the claim holds. 
    Set $\omega = \frac{\log (1 / \rho)}{2(q-2)}$, so that $\beta=\rho(1+\frac{1}{\omega})$.
    We decompose 
    \[ 1 + \rho d X = \left( 1 + \rho d X \right) \cdot 1_{\rho |d X| \geq \omega} + \left( 1 + \rho d X \right) \cdot 1_{\rho |d X| < \omega}, \]
    to obtain 
    \[ \lVert  1 + \rho d X \rVert_q^q = \lVert  \left( 1 + \rho d X \right) 1_{\rho |d X| \geq \omega} \rVert_q^q + \lVert  \left( 1 + \rho d X \right) 1_{\rho |d X| < \omega} \rVert_q^q. \]
    
    \noindent
    When $\rho |d X| \geq \omega$, we have $\omega + \omega \rho |d X| \leq (1 + \omega) \rho |d X|$ and therefore 
    \[ \left| 1 + \rho d X \right| \leq \frac{1 + \omega}{\omega} \rho |d X|.\]
    This allows us to upper bound 
    \[
    \lVert 1 + \rho d X\rVert_{q}^{q} \le \left( \rho \left( 1 + \frac{1}{\omega} \right) \right)^q \lVert  d X\rVert_{q}^{q} + \lVert \left(1 + \rho d X\right) 1_{\rho |d X| <\omega}\rVert_{q}^{q} = \beta^q \lVert  d X\rVert_{q}^{q} + \lVert \left(1 + \rho d X\right) 1_{\rho |d X| <\omega}\rVert_{q}^{q}.
    \]
    Hence, it remains to upper bound the second summand by $\lVert 1+dZ\rVert_q^q$. To achieve this, we first prove the following claim:
    \begin{claim}\label{claim:bound sec_summ general}
        In the above notations, for any $x \in \mathbb{R}$ we have
        \begin{equation}\label{eq:bound_sec_summ genera1}
            \left| 1 + \rho d x \right|^{q} \cdot 1_{\rho |d x| < \omega} \leq 1 + q \rho d x + \frac{\rho^{-1/2}}{2}(q \rho d x)^2.
        \end{equation}
    \end{claim}

    \begin{proof}
        For $w \in \mathbb{R}$ and a positive integer $i$ we write $\binom{w}{i}$ for $\frac{w (w-1) \cdots (w-i+1)}{i !}$. By taking the Taylor approximation of degree $1$ of the function $|1 + y|^q$ at $y = 0$, we may write
        \[
        |1 + y|^q = 1 + qy + \binom{q}{2} |1 + y'|^{q-2} y^2,
        \]
        for some $y'$ in the interval between 0 and $y$. We may upper bound
        \[
        |1 + y'|^{q-2} \leq (1 + |y|)^{q-2},
        \]
        and when substituting $y = \rho d x$ we get
        \[
        (1 + \rho |d x|)^{q-2} \leq (1 + \omega)^{q-2} \leq e^{\omega (q - 2)} = \rho^{-1/2}.
        \]
        We therefore obtain that in the domain $\rho |d x| < \omega$,
        \[
        |1 + \rho d x|^q \leq 1 + q \rho d x + \frac{\rho^{-1/2}}{2}(q \rho d x)^2 .
        \]
        
        We are left with the case $\rho |d x| \geq \omega$. In this case, the left hand side of~\eqref{eq:bound_sec_summ genera1} is equal to zero, while the right hand side 
        $1 + (q \rho d x) + \frac{\rho^{-1/2}}{2} ( q \rho d x )^{2}$ is positive since for $A>1/4$, we have $1 + x + A x^2 > 0$ for all $x$. 
    \end{proof}
    
    We now return to the proof of the lemma. Since $\E[X] = 0$ and $\E[X^2] = 1$, Eq.~\eqref{eq:bound_sec_summ genera1} implies 
    \[
    \lVert  \left(1 + \rho d X\right)1_{\rho |d X| < \omega} \rVert_{q}^{q} \le 1 + \frac{\rho^{1.5}}{2} q^2 d^{2} \leq 1 + \frac{1}{9} q^2 d^2.
    \]
    On the other hand, as $Z$ is symmetric, all of its odd moments vanish. Combining with the facts that $\E[1+dZ]=1$, and $q > 2$, we may use the Binomial theorem to obtain 
    \[
        \lVert 1+dZ\rVert_{q}^{q} \ge \lVert 1+dZ\rVert_{\lfloor q \rfloor}^{\lfloor q \rfloor} \geq \E[(1+dZ)^{\lfloor q \rfloor}] \geq 1 + \binom{\lfloor q \rfloor}{2} d^2 \ge 1 + \frac{1}{9} q^2 d^2,
    \]
    which completes the proof.
\end{proof}

\subsubsection{Linear functions for a general $k$}

In this case, we may write $f=\alpha_0+\sum_{i=1}^{k-1} \alpha_i f_i$ and hence, the assertion of Proposition~\ref{lem:tenzorization} reduces to the following lemma. 
    \begin{lem}
    \label{lem:n=1_many_vars}
    Let $\rho \leq 1/3$, $q \ge 2$, and let $f \in L^2(\Omega, \mu)$ be a linear function. Then
   \[ \left\lVert  \encnoise f \right\rVert_{q}^{q} \le \left\lVert  \enc f \right\rVert_{q}^{q} + \beta^{q} \left\lVert  L f \right\rVert_q^q,\]
   where $\beta=\rho \left(1 + \frac{2 (q-2)}{\log(1/\rho)}\right)$.
    \end{lem}
    
    \begin{proof}
    Like in the case $k=2$, we may assume without loss of generality that $\alpha_0=1$, and hence, $f = 1 + \sum_{i=1}^{k-1} \alpha_i f_i$. In this case, $\lVert G(f)\rVert_q^q=\lVert 1+\sum_{i=1}^{k-1} \alpha_i Z_i\rVert_q^q,$ where $Z_i$ are independent normal Gaussian random variables. 
    Denote $X = L(f) = \sum_{i=1}^{k-1} \alpha_i f_i$, $\sigma = \sqrt{\E[X^2]}$, and note that $\encnoise (f) = 1 + \rho X = 1 + \rho \sigma \frac{X}{\sigma}$.
    We have $\E \left[ \frac{X}{\sigma} \right] = 0$ and $\E \left[ \left( \frac{X}{\sigma} \right)^2 \right] = 1$.
    By Lemma~\ref{lem:n=1_one_var_general} (applied with $d = \sigma$ and $\frac{X}{\sigma}$ in place of $X$), for any normal Gaussian variable $Z$ we have, 
    \begin{equation*}
        \left\lVert  T_{\rho} f \right\rVert_q^q = \left\lVert 1+\sigma \frac{X}{\sigma} \right\rVert_q^q \le \left\lVert  1 + \sigma Z \right\rVert_{q}^{q} + \beta^{q} \left\lVert  \sigma \frac{X}{\sigma} \right\rVert_q^q = \left\lVert  1 + \sigma Z \right\rVert_{q}^{q} + \beta^{q} \left\lVert  L f \right\rVert_q^q.
    \end{equation*}  
    The assertion of the lemma follows, as $\sigma Z$ has the same distribution as $\sum_{i=1}^{k-1} \alpha_i Z_i$. 
    \end{proof}

\subsubsection{General functions}

Now we are ready to prove Proposition~\ref{lem:tenzorization}.

\begin{proof}
    Let $T_\rho^{(i)}$ be the operator $f \mapsto E_i[f] + \rho L_i[f]$. We may write $T_\rho$ as the composition of the operators $T_\rho^{(i)}$, since 
    \[
    T_{\rho}^{(i)} = (I \otimes \cdots \otimes \underbrace{T_\rho}_i \otimes \cdots \otimes I).
    \]
We claim that for any function $g$ we have
\begin{equation}\label{eq: applying one dimensional hypercontractivity}
\lVert T_{\rho}^{(i)} g \rVert_{q}^q \le \lVert G_i g\rVert_q^q+ \beta^q\lVert L_i g\rVert_q^q. 
\end{equation}
We will then apply this assertion repeatedly (for $i=1,\ldots,n$) to complete the proof.

To prove~\eqref{eq: applying one dimensional hypercontractivity}, we may assume without loss of generality that $i=1$.  We may choose a random input  $x\sim \mu$ and write it as $(x_1,y),$ where $y$ denotes the last $n-1$ coordinates of $x$. Let $\mathrm{ev}_y\colon L^2(\Omega^{\{2,\ldots, n\}})\to \mathbb{R}$ be the `evaluation at $y$' operator given by $g\mapsto g(y)$. Then the restriction operator  $g\mapsto g_{\{2,\ldots , n\}\to y}$ takes the form $I\otimes \mathrm{ev_y}$. This shows that the restriction operator commutes with $T_{\rho}^{(1)},E_1,$ and $L_1$.
We may now apply Lemma~\ref{lem:n=1_many_vars} to the restricted function $g_y:=g_{\{2,\ldots, n\}\to y}$ and take expectations to obtain: 
    
        \begin{align*}
            \lVert T_{\rho}^{(i)} g \rVert_{q}^q = \mathbb{E}_x\left[\left |T_{\rho}^{(1)} g(x)\right |^q\right]
            = \mathbb{E}_y\left[\lVert  T_{\rho}^{(1)} \left(g_{y} \right)\rVert_q^q\right]  
             \leq 
            \EE_{y} \left[ \EE \left[ |\enc_{\{1\}}(g_y)|^q  + \beta^q |L_1 g_y|^q   \right] \right] = 
            \lVert  \enc_{1} g \rVert_q^q + \beta^q \lVert  L_1 g \rVert_q^q,
        \end{align*}
        where the last equality follows from Fubini's theorem. This proves~\eqref{eq: applying one dimensional hypercontractivity}.

\medskip

        We now complete the proof by showing inductively that for every $m\le n$, we have 
        \begin{equation}\label{eq:induction}
        \lVert T_{\rho}^{(m)}\circ T_{\rho}^{(m-1)}\circ \cdots \circ T_{\rho}^{(1)}f\rVert_q^q \le \sum_{S\subseteq [m]}\beta^{q|S|}\lVert \left(L_S\circ G_{[m]\setminus S}\right)f\rVert_q^q.  
        \end{equation}
        This will complete the proof by setting $m=n$. 
        
        For $m=1$, the claim is exactly~\eqref{eq: applying one dimensional hypercontractivity}, with $g=f$ and $i=1$. Let us assume that~(\ref{eq:induction}) holds for $m-1$ and prove it for $m$. Write 
        \[
        T_{\rho}^{[m-1]}=  T_{\rho}^{(m-1)}\circ \cdots \circ T_{\rho}^{(1)}.
        \] 
        We may apply (\ref{eq: applying one dimensional hypercontractivity}) and the commutative properties of tensor products to obtain that 
        \begin{align*}
        \lVert T_{\rho}^{(m)}T_{\rho}^{[m-1]}f\rVert_q^q  
        \le
        \lVert  G_{m} T_{\rho}^{[m-1]} f\rVert_q^q  + \beta^q \lVert  L_m  T_{\rho}^{[m-1]} f\rVert_q^q 
        = \lVert  T_{\rho}^{[m-1]} G_{m} f\rVert_q^q  + \beta^q \lVert   T_{\rho}^{[m-1]} L_m f\rVert_q^q.        
        \end{align*}
        Applying the induction hypothesis with $L_m(f)$ and $G_{m}(f)$ in place of $f$ completes our inductive proof. 
\end{proof}

\section{Sharp Hypercontractive Inequalities for Global Functions on $\Omega^n$}
\label{sec:hypercontractive}

In this section we prove sharp hypercontractive inequalities for functions on $\Omega^n$. First, we prove a hypercontractive inequality for general functions, and then we use it to prove a more powerful hypercontractive inequality for global functions. Finally, we use the latter inequality to prove a slightly stronger version of Theorem~\ref{cor:hypercontractivity for global functions - intro}.

\subsection{Hypercontractivity for general functions}

We prove the following theorem.
\begin{thm}
\label{thm:bounding_noised_q_norm_with_derivatives}
Let $\rho \leq 1/3$ and $q \ge 2$. Then for any $f \in L^2(\Omega^n,\mu)$,
        \[ \lVert  \totnoise f \rVert_q^q \leq \sum_{S \subseteq [n]} \beta^{q |S|} q^{-q|S|/2} \EE_{x \sim \Omega^S} \lVert  D_{S, x} f \rVert_2^q ,\]
where $\beta=\rho \Big(1 + \frac{2 (q-2)}{\log(1/\rho)}\Big)$.
\end{thm}
In the proof we rely on the classical hypercontractive inequality in Gaussian space proved by Bonami~\cite{Bonami}, Gross~\cite{gross}, and Beckner~\cite{Beckner}.
\begin{thm}\label{thm: hypercontractivity in Gaussian space}
Let $q\ge p \ge 1$ and let $\rho \le \sqrt{\frac{p-1}{q-1}}.$ For any $f\in L^p(\mathbb{R}^n,\gamma)$, we have
$\lVert T_\rho f\rVert_q\le \lVert f\rVert_p$. 
\end{thm}
We also use the following simple commutativity properties of the operators $G_S,T_\rho,$ and $D_{S, x}$ which complement the `swapping rules' of $L_S$ and $D_{S,x}$ with $T_\rho$ that were presented in Lemma~\ref{lem:swap laplace noise derivative}.
\begin{lem}\label{lem:Basic Fourier formulas}
\begin{enumerate}
    
    \item For any $S,T \subset [n]$ we have $T_{\rho} G_S = G_{S}T_{\rho}$.
    \item If $S,T$ are disjoint, then $G_T D_{S,x} = D_{S,x}G_T$.
\end{enumerate}
\end{lem}
\begin{proof}
We make use of the following tensor descriptions of our operators. 
\[D_{S, x} = \bigotimes_{i \in S} D_{i,x_i} \otimes \bigotimes_{i \in S^c} I, \qquad \qquad T_{\rho} = \bigotimes_{i=1}^n T_{\rho}^{(i)}, \qquad \mbox{ and } \qquad \enc_S = \bigotimes_{i\in S} \enc_i\otimes \bigotimes_{i\in S^c} I. \]
As the composition of operators is performed coordinate-wise,~(2) follows immediately from the tensor description, and~(1) reduces to the case $n=1$, where it follows easily from the formula $T_{\rho}f=f^{=\emptyset}+\rho f^{=\{1\}}$ (which is the $n=1$ case of $T_{\rho}f=\sum_S \rho^{|S|}f^{=S}$) and the formula $G(\alpha_0+\sum_{i=1}^{k-1}\alpha_i x_i)=\alpha_0+\sum_{i=1}^{k-1} \alpha_i z_i$ presented above. 
\end{proof}

Now we are ready to present the proof of 
Theorem~\ref{thm:bounding_noised_q_norm_with_derivatives}.

    \begin{proof}
By Proposition~\ref{lem:tenzorization} and Lemmas~\ref{lem:swap laplace noise derivative},~\ref{lem:Basic Fourier formulas} we have      
        \begin{align*}
            \lVert  \totnoise f \rVert_q^q 
            & = \lVert  \encnoise \contrnoise f \rVert_q^q  
            \leq 
            \sum_{S \subseteq [n]} \beta^{q |S|} \lVert  (L_S \circ \enc_{S^c}) \contrnoise f \rVert_q^q  
             \\ & =
            \sum_{S \subseteq [n]} \beta^{q |S|} \EE_{x \sim \Omega^S} \left\lVert  \left[ L_S \left( \enc_{S^c} \contrnoise f \right) \right]_{S \rightarrow x} \right\rVert_q^q  
            = 
            \sum_{S \subseteq [n]} \beta^{q |S|} \EE_{x \sim \Omega^S} \lVert  \enc_{S^c} D_{S, x} \contrnoise f \rVert_q^q  
             \\ & =
            \sum_{S \subseteq [n]} \beta^{q |S|} \EE_{x \sim \Omega^S} \left\lVert  \left( \contrnoiseamount \right)^{|S|} \enc_{S^c} \contrnoise D_{S, x} f \right\rVert_q^q  \label{eq:swap_op}
            = 
            \sum_{S \subseteq [n]} \beta^{q|S|} q^{\frac{-q |S|}{2}} \EE_{x \sim \Omega^S} \lVert  \enc_{S^c} \contrnoise D_{S, x} f \rVert_q^q  
             \\ & =
            \sum_{S \subseteq [n]} \beta^{q|S|} q^{-\frac{q |S|}{2}} \EE_{x \sim \Omega^S} \lVert  \contrnoise \left( \enc_{S^c} D_{S, x} f \right) \rVert_q^q .  
        \end{align*}
        The function $\enc_{S^c} D_{S, x} f$ is a function on the Gaussian space $(\mathbb{R}^{k-1})^{S^c}$, and hence, we may apply Theorem \ref{thm: hypercontractivity in Gaussian space} to obtain 
        \begin{align*}
            \lVert  \totnoise f \rVert_q^q 
            &\leq \sum_{S \subseteq [n]} \beta^{q|S|} q^{-\frac{q |S|}{2}} \EE_{x \sim \Omega^S} \lVert  \contrnoise \left( \enc_{S^c} D_{S, x} f \right) \rVert_q^q
            \leq
            \sum_{S \subseteq [n]} \beta^{q|S|} q^{-\frac{q |S|}{2}} \EE_{x \sim \Omega^S} \lVert  \enc_{S^c} D_{S, x} f \rVert_2^q \\
            &=
            \sum_{S \subseteq [n]} \beta^{q|S|} q^{-\frac{q |S|}{2}} \EE_{x \sim \Omega^S} \lVert  D_{S, x} f \rVert_2^q,
        \end{align*}
        where the last equality holds since $G_{S^c}$ preserves 2-norms. This completes the proof.
    \end{proof}

\subsection{Hypercontractivity for global functions}

In this section, we will need a `derivative-based' notion of globalness, which we now define. The transition from this notion to the `restriction-based' notion used in the statement of our main theorems is established in Lemma~\ref{lem:restriction global implies global}.
\begin{definition} \label{def: global}
    Let $r,\gamma>0$. A function $f \in L^2(\Omega^n,\mu)$ is called \emph{$(r,\gamma)$-$L_2$-global} if $\lVert D_{S,x}f\rVert_2\le r^{|S|}\gamma$ for all $S\subseteq [n]$ and $x\in \Omega^S.$
\end{definition}

We prove the following theorem.
\begin{thm} \label{thm:hypercontractivity for global functions - alternative}
     Let $r, \gamma>0,$ and let $q \ge 2$. Let $f\colon \Omega^n\to \mathbb{R}$ be an $(r ,\gamma)$-$L_2$-global function. Let $\rho = \frac{1}{\sqrt{2q}} \rho'$ and suppose that $\rho' \leq 1/3$ and that $\beta = \rho' \left( 1 + \frac{2(q-2)}{\log(1 / \rho')} \right)$ satisfies
     \begin{equation}\label{Eq:Cond-rho'}
        \beta \leq \sqrt{q} (r/\sqrt{2})^{-\frac{q-2}{q}}.
     \end{equation}
      Then 
     \[\lVert T_{\rho}f\rVert_q^q \le \lVert f\rVert_2^2\gamma^{q-2}.\]
 \end{thm}
In the proof we use the following simple lemma.
\begin{lem}\label{lem: Efron Stein formula}
For any $f\colon \Omega^n\to \mathbb{R},$ we have \[\sum_{S\subseteq [n]} \lVert  L_S T_{1/\sqrt{2}}[f]\rVert_2^2 = \lVert f\rVert_2^2.
\]
 \end{lem}
 \begin{proof}
     By Lemma~\ref{lem:laplacian_efron_stein} and~\eqref{eq:noise_efron_stein}, we have 
     \[ L_S T_{1/\sqrt{2}}[f] = \sum_{T\supseteq S} 2^{-|T|/2} f^{=T}.\]
     Taking 2-norms, we obtain 
     \[\sum_{S\subseteq [n]} \lVert  L_S T_{1/\sqrt{2}}[f]\rVert_2^2 = \sum_{S\subseteq [n]}\sum_{T\supseteq S} 2^{-|T|}\lVert f^{=T}\rVert_2^2= \sum_{T\subseteq [n]}\lVert f^{=T}\rVert_2^2 =  \lVert f\rVert_2^2,\]
     as asserted.
 \end{proof}

 \begin{proof}[Proof of Theorem~\ref{thm:hypercontractivity for global functions - alternative}]
     Write $f'=T_{\frac{1}{\sqrt{2}}} f.$ By Theorem~\ref{thm:bounding_noised_q_norm_with_derivatives}, we have 
     \begin{equation}\label{Eq:hyper-aux1}
     \lVert T_{\rho}f\rVert_q^q = \lVert T_{\rho' \frac{1}{\sqrt{q}}}f' \rVert_q^q \le \sum_{S\subseteq [n]} \left( \frac{\beta}{\sqrt{q}} \right)^{q|S|} \mathbb{E}_{x\sim \Omega^S} \lVert  D_{S,x}[f']\rVert_2^q .     
     \end{equation}
    By Lemma~\ref{lem:swap laplace noise derivative}, for any $S \subset [n]$ and any $x \in \Omega^S$, we have $D_{S,x}[f']=D_{S,x}T_{1/\sqrt{2}}f = \left( \frac{1}{\sqrt{2}} \right)^{|S|} T_{1/\sqrt{2}} D_{S, x}[f]$, and hence, 
    \begin{align}\label{Eq:hyper-aux2}
    \begin{split}
    \lVert  D_{S,x}[f']\rVert_2^{q-2} &= \lVert  \left( \frac{1}{\sqrt{2}} \right)^{|S|} T_{1/\sqrt{2}} D_{S, x}[f] \rVert_2^{q-2} =
    \left( \frac{1}{\sqrt{2}} \right)^{|S|(q-2)}
    \lVert T_{1/\sqrt{2}} D_{S,x} f\rVert_2^{q-2} \\ &\le \left( \frac{1}{\sqrt{2}} \right)^{|S|(q-2)}\lVert D_{S,x} f\rVert_2^{q-2} \le ((r/\sqrt{2})^{|S|}\gamma)^{q-2},     
    \end{split}
    \end{align}
    where the penultimate inequality holds since $T_{1/\sqrt{2}}$ is a contraction and the last inequality holds as $f$ is $(r ,\gamma)$-$L_2$-global. Since $\mathbb{E}_{x\sim \Omega^S} \lVert  D_{S,x}[f']\rVert_2^q \leq \max_{x\in \Omega^S} \lVert  D_{S,x}[f']\rVert_2^{q-2} \mathbb{E}_{x\sim \Omega^S} \lVert  D_{S,x}[f']\rVert_2^2$, we may combine~\eqref{Eq:hyper-aux1} and~\eqref{Eq:hyper-aux2} to obtain  
     \[
     \lVert T_{\rho} f\rVert_q^q \le \sum_{S\subseteq [n]} \left( \left( \frac{\beta}{\sqrt{q}} \right)^q \left( \frac{r}{\sqrt{2}} \right)^{q-2} \right)^{|S|} \gamma^{q-2} \mathbb{E}_{x\sim \Omega^S} \lVert D_{S,x} T_{\frac{1}{\sqrt{2}}} f\rVert_2^2 .
     \]
    By~\eqref{Eq:Cond-rho'}, we have 
    \[
    \left( \frac{\beta}{\sqrt{q}} \right)^q \left( \frac{r}{\sqrt{2}} \right)^{q-2} \le 1,
    \]
    and thus,
    \[
    \lVert T_{\rho} f\rVert_q^q \le \sum_{S\subseteq [n]} \gamma^{q-2} \mathbb{E}_{x\sim \Omega^S} \lVert D_{S,x} T_{\frac{1}{\sqrt{2}}} f\rVert_2^2 = \gamma^{q-2} \sum_{S\subseteq [n]} \lVert  L_S T_{1/\sqrt{2}} f\rVert_2^2 = \gamma^{q-2} \lVert f \rVert_2^2,
    \]
    where the last equality uses Lemma \ref{lem: Efron Stein formula}. This completes the proof.
\end{proof}

As the assertion of Theorem~\ref{thm:hypercontractivity for global functions - alternative} is somewhat cumbersome, we provide a weaker, yet useful, corollary.
\begin{cor} \label{cor:hypercontractivity for global functions - alternative}
     Let $r, \gamma>0,$ and let $q \ge 2$. Let $f\colon \Omega^n\to \mathbb{R}$ be an $(r ,\gamma)$-$L_2$-global function. If  
    \begin{enumerate}
        \item $\rho \le \frac{1}{3 \sqrt{2}} \min\left(\frac{1}{r^{\frac{q-2}{q}}q},\frac{1}{\sqrt{q}}\right)$, or

        \item $r \geq 1$ and $\rho \leq \frac{\log q}{16 r q}$,
    \end{enumerate}
then $\lVert T_{\rho}f\rVert_q^q \le \lVert f\rVert_2^2\gamma^{q-2}$.
 \end{cor}

\begin{proof}
    Set $\rho' = \sqrt{2q} \cdot \rho$ and $\beta = \rho' \left( 1 + \frac{2(q-2)}{\log(1 / \rho')} \right)$. It is clear that if either~(1) or~(2) holds then $\rho'<1/3$. Hence, in order to prove that Corollary~\ref{cor:hypercontractivity for global functions - alternative} follows from Theorem~\ref{thm:hypercontractivity for global functions - alternative}, it is sufficient to show that $\beta$ satisfies~\eqref{Eq:Cond-rho'}. 

    Since for any $q,\rho'>0$, 
    \[
        \frac{1}{3} \left( 1 + \frac{2q}{\log(1 / \rho')} \right) \le \frac{q}{\min \{q, \log(1 / \rho') \} },
     \]
    and as $q \ge 2$, $\rho' \le 1/3$, we have
    \[
        \beta = (3\rho') \cdot \frac{1}{3} \left( 1 + \frac{2(q-2)}{\log(1 / \rho')} \right) \le \frac{3 \rho' q}{\min \{q, \log (1 / \rho') \}} \le 3 \rho' q.
     \]
     If~(1) holds, then in particular, $\rho \le \frac{1}{3 \sqrt{2} r^{\frac{q-2}{q}} q}$, and thus,
     \[
     \beta \le 3 \rho' q \le 3 \frac{\sqrt{2q}}{3 \sqrt{2} r^{\frac{q-2}{q}} q} q \le \frac{\sqrt{q}}{r^{\frac{q-2}{q}}}, 
     \]
     which proves~\eqref{Eq:Cond-rho'}. 

    If~(2) holds (that is, if $r > 1$ and $\rho \leq \frac{\log q}{16 r q}$), then we have 
     \[
     \rho' \le \frac{1}{8 \sqrt{2}} \frac{\log q}{r \sqrt{q}} \le \frac{1}{q^{1/3}},
     \]
     since $q \ge 2$, and thus, $\min \{q, \log (1 / \rho') \} \ge \frac{1}{3} \log q$. Hence, 
     \[
     \beta \leq \frac{3 \rho' q}{\min \{q, \log (1 / \rho') \}} \leq \frac{3 \frac{1}{8 \sqrt{2}} \frac{\log q}{r \sqrt{q}} q}{\frac{1}{3} \log q} \le \frac{\sqrt{q}}{r}, 
     \]
     which implies~\eqref{Eq:Cond-rho'} since $r > 1$. This completes the proof. 
\end{proof}


\subsection{Proof of Theorem~\ref{cor:hypercontractivity for global functions - intro}} 
\label{sec: proof of hyper for glob}

    We use the following `derivative-based' notion of globalness, which generalizes Definition~\ref{def: global}.
    \begin{definition}
        For $p \geq 1$, we say that $f$ is \emph{$(r,\gamma,d)$-$L_p$-global} if $\lVert D_{S, x}f \rVert_p \le r^{|S|}\gamma$ for all $S\subseteq [n]$ of size $\le d$ and all $x\in \Omega^S.$
    \end{definition}
    The following easy lemma shows that the `restriction-based' notion of globalness 
    implies the `derivative-based' notion, with almost the same parameters.
    \begin{lem}\label{lem:restriction global implies global}
        Let $p \geq 1$, and let $f \colon \Omega^n \to \mathbb{R}$. Suppose that for some $r \geq 1,\gamma>0$, we have $\lVert f_{S\to x}\rVert_p \le r^{|S|}\gamma$ for all $S$ of size $\le d$ and all $x\in \Omega^S.$ Then $f$ is $(2r,\gamma,d)$-$L_p$-global. 
    \end{lem}
    \begin{proof}
        Let $S \subset [n]$, $|S| \leq d$. As was shown in Section~\ref{sec: background}, the Laplacian $L_S[f]$ satisfies 
        \[
        L_S[f]=\sum_{T\subseteq S}(-1)^{|T|} E_T[f].
        \]
        By the triangle inequality, this implies $\lVert D_{S,x} f\rVert_p\le \sum_{T\subseteq S}\lVert [E_{T}[f]]_{S\to x}\rVert_p.$ For each $T \subset S$, we obtain by Jensen's inequality that 
        \[
        \lVert [E_{T}[f]]_{S\to x}\rVert_p^p = \EE_{ z\sim \Omega^{S^c}} \left[\left|\EE_{y\sim \Omega^T} f_{S\setminus T \to x_{S\setminus T}, S^c\to z}\right|^p\right] \le \lVert  f_{S\setminus T \to x_{S\setminus T}}\rVert_p^p\le r^{p|S|}\gamma^p. 
        \]
        Hence,
        \[
        \lVert D_{S,x} f\rVert_p\le \sum_{T\subseteq S}\lVert [E_{T}[f]]_{S\to x}\rVert_p \leq \sum_{T\subseteq S} r^{|S|}\gamma = (2r)^{|S|}\gamma,
        \]
        as asserted.
    \end{proof}

    This allows us to obtain our main hypercontractivity theorem (Theorem~\ref{cor:hypercontractivity for global functions - intro}) from Corollary~\ref{cor:hypercontractivity for global functions - alternative}.
    We prove a stronger formulation with an additional parameter $\gamma$, which reduces to the statement of Theorem~\ref{cor:hypercontractivity for global functions - intro} by substituting $\gamma=\lVert f\rVert_2$.
    


\begin{thm}
\label{thm: strong hypercontractivity for global functions}
 Let $q \ge 2$, $\gamma>0$, and let $(\Omega,\mu)$ be a finite probability space. Let $f\colon (\Omega^n, \mu^n) \to \mathbb{R}$, and assume that \[\lVert f_{S \to x}\rVert_2 \leq r^{|S|} \gamma\] for all $S \subseteq [n]$ and for all $x \in \Omega^S$. If $r \geq 1$ and $\rho \leq \frac{\log q}{32 r q}$,
 then 
 \[
 \lVert T_{\rho}f\rVert_q^q \le \gamma^{q-2} \lVert f\rVert_2^2.
 \]
\end{thm}

    \begin{proof}
        By Lemma~\ref{lem:restriction global implies global}, $f$ is $(2r, \gamma)$-$L_2$-global. Hence, the assertion of the theorem follows from Corollary~\ref{cor:hypercontractivity for global functions - alternative}.
    \end{proof}

    \medskip The following proposition shows that Theorem~\ref{cor:hypercontractivity for global functions - intro} is tight, up to the value of the constant $C$.
    \begin{proposition}
        \label{lem: global hyper sharpness}
        There exist $r,C>1$ such that for any $q \geq 2$, there exist $n,p$ and a function $f \colon (\{0,1\}^n,\mu_p) \to \mathbb{R}$, such that
        
        \medskip \text{(a)}: \quad $\lVert f_{S \to x}\rVert_2 \leq r^{|S|} \lVert f\rVert_2$ for all $S \subset [n]$, and for all $x \in \bool^S$, and
        
        \medskip \text{(b)}: \quad $\lVert T_{\rho}f\rVert_q > \lVert f\rVert_2$, for all $\rho > \frac{C\log q}{r q}$.
    \end{proposition}

    \begin{proof}
          For $n,d \in \mathbb{N}$ such that $d|n$, let $p=d/n$ and consider the function $f_{n,d} \colon (\{0,1\}^n, \mu_{p}) \to \mathbb{R}$ defined by
            \[f_{n,d}= \prod_{j=0}^{d-1}\left(\sum_{i=1}^{n/d} (x_{j(n/d)+i}-p) \right).\] 
          For any $n,d$, the function $f_{n,d}$ satisfies~(a) with $r=2$. Indeed, if the set $S$ of `restricted' coordinates is included in a single expression of the form $f_{n,d,j}=\sum_{i=1}^{n/d} (x_{j(n/d)+i}-p)$, the 2-norm increases the most when all the coordinates are fixed to $1$, and in this case we have
            \[
                \lVert  (f_{n,d,j})_{S \rightarrow 1} \rVert_2^2 = (|S| (1-p))^2 + (\frac{n}{d} - |S|) (p(1-p)) \le |S|^2 + 1,
            \]
          and $\lVert  f_{n,d,j}\rVert_2 = \sqrt{1-p}$. As $\sqrt{\frac{|S|^2 + 1}{1-p}} \leq 2^{|S|}$ for all $|S| \geq 1$, condition~(a) is satisfied for any such $S$. The condition for an arbitrary set $S$ follows by multiplicativity.

          When $d$ is fixed and $n$ tends to infinity, the distribution of $f_{n,d}$ converges to the distribution of $\prod_{i=1}^{d}(X_i-1),$ where each $X_i$ is independently $\mathrm{Poi}(1)$ distributed. This easily implies that when $q$ is large and fixed and $n \to \infty$, the $q$-norm of $f_{n,d}$ tends to $\left(\Theta(\frac{q}{\log q})\right)^d$, while its $2$-norm is $(1-p)^{d/2}.$ Moreover, for any $\rho>0$, $f_{n,d}$ is an eigenfunction of $T_\rho$ corresponding to the eigenvalue $\rho^d$. Hence, there exist $C'>0$ and $n_0=n_0(q,d)$, such that $\lVert T_{\frac{C'\log q}{q}}f_{n,d}\rVert_q>\lVert f_{n,d}\rVert_2$, for all $n>n_0$. In particular, condition~(b) holds for $f_{n,d}$, for all $C\geq 2C'$. Therefore, the function $f=f_{n,d}$ for any $n>n_0$ satisfies the assertion of the proposition, with $r=2$ and $C=2C'$.
     \end{proof}

\section{Sharp Level-$d$ Inequalities for Global Functions}
\label{sec:level-d}

In this section we prove our level-$d$ inequality -- namely, Theorem~\ref{thm:level d intro}, which asserts a sharp bound on the 2-norm of $f^{=d}$ for `global' functions $f \colon \Omega^n \to \{0,1\}$. Here, the `globalness' notion concerns \emph{restrictions} of $f$ -- it asserts that  there exist $r,\gamma$ such that for any $S \subset [n]$ with $|S| \leq d$, $\EE[f_{S \to x}] \le r^{|S|}\gamma$. 

First, we prove a sharp level-$d$ inequality under the assumption that $f^{=d}$ is global. Then, we pass via a derivative-based notion of globalness -- namely, the assumption that $\lVert D_{S,x} f\rVert_p \leq r^{|S|}\gamma_p$ for $p=1,2$, to show that globalness of $f$ can be upgraded to globalness of $f^{=d}$, thus allowing us to prove the theorem.  

Throughout this section, we use the convention $(\cdot)^0=1$ for any expression $(\cdot)$, including expressions of the form $\frac{r}{0}$ and $\log 1$.

\subsection{A level-$d$ inequality when $f^{=d}$ is global}

\begin{lem}\label{lem:level d given globalness}
    Let $\gamma>0,\gamma_2>\gamma_1>0,$ let $f\colon \Omega^n\to \mathbb{R}$ be with $\lVert f\rVert_2 \le \gamma_2$ and $\lVert f\rVert_1\le \gamma_1$, let  $d<\frac{1}{2}\log\left(\frac{\gamma_2}{\gamma_1}\right)$, and let $r\ge\left(\frac{d}{\log(\gamma_2/\gamma_1)}\right)^{1/2}$ .
    If the function $f^{=d}$ is $(r,\gamma)$-$L_2$-global,  then 
    \[ \lVert f^{=d}\rVert_2^2\le \left(\frac{33r}{d}\right)^d \gamma_1 \cdot \gamma\log^{d}\left(\frac{\gamma_2}{\gamma_1}\right). \]
\end{lem}
\begin{proof}
    The assertion clearly holds for $d=0$. Hence, we may assume $d \geq 1$.    
    Let $q>2,q'$ be H\"{o}lder conjugates to be chosen later. Then by H\"{o}lder's inequality we have 
    \begin{align*}
        \lVert f^{= d}\rVert_2^2 = \langle f, f^{= d }\rangle \le \lVert f^{= d}\rVert_{q}\lVert f\rVert_{q'}.
    \end{align*}
    Let $\rho = \frac{1}{3 \sqrt{2}} \min\left(\frac{1}{r^{\frac{q-2}{q}}q},\frac{1}{\sqrt{q}}\right)$. By Corollary~\ref{cor:hypercontractivity for global functions - alternative}, we have 
    \[
       \lVert f^{= d}\rVert_q = \rho^{-d}\lVert \mathrm{T}_{\rho} f^{=d}\rVert_q \le  \rho^{-d}\lVert f^{=d}\rVert_2^{2/q}\gamma^{1-2/q} \le \rho^{-d}\gamma.      
 \]
Let $\theta=\frac{2}{q}$. Since $\frac{1}{q'}= 1-\frac{1}{q}=\frac{1-\theta}{1}+\frac{\theta}{2},$ by the log-convexity of $L_p$-norms we have   
    \[ 
    \lVert f\rVert_{q'} \le \lVert f\rVert_1^{1-\theta}\lVert f\rVert_2^{\theta}\le \gamma_1 \left(\frac{\gamma_2}{\gamma_1}\right)^\theta.
    \]
Combining with the above bounds, we obtain
    \[
    \lVert f^{= d}\rVert_2^2 \leq \lVert f^{= d}\rVert_{q}\lVert f\rVert_{q'} \leq 
    \rho^{-d}\gamma \cdot \gamma_1 \left(\frac{\gamma_2}{\gamma_1}\right)^\theta.
    \]
In order to optimize the right hand side, we choose $q=2\log(\gamma_2/\gamma_1)/d$, thus obtaining $\theta=d/\log(\gamma_2/\gamma_1)$, which yields 
\begin{equation}\label{Eq:d-level-Aux1}
    \lVert f^{= d}\rVert_2^2 \leq  
    \rho^{-d}\gamma \cdot \gamma_1 \left(\frac{\gamma_2}{\gamma_1}\right)^\theta = (e/\rho)^{d}\gamma_1\cdot \gamma. 
    \end{equation}
Note by the assumptions on $r$ and $d$, we have $q>4$ and $r>1/\sqrt{q}$. It follows that 
\[\frac{1}{3\sqrt{2}r^{\frac{q-2}{q}}q} = \frac{1}{3\sqrt{2}rq} \cdot r^{2/q}>\frac{1}{6rq}
\]
(since $r^{2/q} > (1/q)^{1/q}$, and for any $x<1/4$ we have $x^x > 1/\sqrt{2}$),
and 
\[\frac{1}{3\sqrt{2q}}>\frac{1}{6r q}, \]
and therefore, $\rho \ge \frac{1}{6rq} = \theta/(12r)$. Substituting into~\eqref{Eq:d-level-Aux1}, we obtain 
\[
\lVert f^{=d}\rVert_2^2 \le \left(\frac{e \cdot 12 r}{\theta} \right)^{d}\gamma_1\cdot \gamma \leq \left(\frac{33r}{d}\right)^d \gamma_1 \cdot \gamma \cdot \log^{d}\left(\frac{\gamma_2}{\gamma_1}\right), 
\]
as asserted.
\end{proof}

\subsection{Obtaining globalness of $f^{=d}$}

The following theorem shows that `derivative-based' globalness of $f$ can be used to obtain globalness of $f^{=d}$. In view of Lemma~\ref{lem:restriction global implies global}, it will be sufficient for proving Theorem~\ref{thm:level d intro}. 
\begin{thm}\label{thm: level d is global}
     Let $r\ge 1, \gamma_2>\gamma_1>0$, and let $d\le \frac{1}{2}\log(\frac{\gamma_2}{\gamma_1})$ be a non-negative integer. Let $f \colon \Omega^n \to \mathbb{R}$ be a function that is both  $(r,\gamma_1,d)$-$L_1$-global and $(r,\gamma_2,d)$-$L_2$-global. Let
    \[
        r'_d = \frac{\sqrt{d}}{\log^{1/2}(\gamma_2/\gamma_1)} \qquad \mbox{and} \qquad \gamma'_d =  \left(\frac{33 r}{\sqrt{d}}\right)^d \gamma_1 \log^{d/2}\left(\frac{\gamma_2}{\gamma_1}\right).
    \]
    Then the function $f^{=d}$ is $(r'_d,\gamma'_d)$-$L_2$-global.
\end{thm}

\begin{proof}
    We prove the theorem by induction on $d$. For $d=0$, the theorem asserts that $f^{=0}$ is $(0,\gamma_1)$-$L_2$-global, and this indeed holds since $f^{=0}=\EE[f]$. 
    
    To prove the induction step, we show that each function $D_{i,x}[f^{=d}]$ is $(r'_d,r'_d\gamma'_d)$-$L_2$ global and that $\lVert f^{=d}\rVert_2\le \gamma'_d.$ This will complete the proof, as each derivative $D_{S,x}[f^{=d}]$ is also a derivative of the function $D_{i,x}[f^{=d}]$ (for any $i \in S$).
    We first show the former claim. 
    
    Each function $D_{i,x}[f]$ is both  $(r,r\gamma_1,d-1)$-$L_1$-global and $(r,r\gamma_2,d-1)$-$L_2$-global. Therefore, we may apply the induction hypothesis to obtain that the function \[D_{i,x}[f^{=d}]= \left(D_{i,x}[f]\right)^{= d - 1}\] is $(r''_{d-1}, \gamma''_{d-1},d-1)$-$L_2$-global, where 
    \[
    r''_{d-1}:=\sqrt{\frac{d-1}{d}}r'_d, \qquad \mbox{and} \qquad
    \gamma''_{d-1}:= \left(\frac{33 r}{\sqrt{d-1}}\right)^{d-1} r\gamma_1\log^{(d-1)/2}\left(\frac{\gamma_2}{\gamma_1}\right).  
    \]
    As we clearly have $r''_{d-1} \le r'_d$, and as 
\[    \gamma''_{d-1} \le  \left(\frac{33r}{\sqrt{d}}\right)^d r'_d \gamma_1 \log^{d/2}\left(\frac{\gamma_2}{\gamma_1}\right) = r'_d\gamma'_d, \]
    $D_{i,x}[f^{=d}]$ is indeed $(r'_d,r'_d\gamma'_d)$-$L_2$ global. (In the case $d=1$, the function $D_{i,x}[f^{=1}]=(D_{i,x}[f])^{=0}=\mathbb{E}[D_{i,x}(f)]$ is $(0,r\gamma_1)$-$L_2$-global, and in particular, is $(r'_1,r'_1\gamma'_1)$-$L_2$ global.)   

It now remains to show that $\lVert f^{=d}\rVert_2 \le \gamma'_d.$ Suppose on the contrary that $\lVert f^{=d}\rVert_2>\gamma'_d$. By the above paragraph, it follows that $f^{=d}$ is $(r'_d,\lVert f^{=d}\rVert_2)$-global. Therefore, by Lemma \ref{lem:level d given globalness} we have 
\[
\lVert f^{=d}\rVert_2^2\le \left(\frac{33 r'_d}{d}\right)^d \gamma_1 \lVert f^{=d}\rVert_2 \log^d\left(\frac{\gamma_2}{\gamma_1}\right).
\]
By substituting the values of $r'_d$ and $\gamma'_d$, this yields
\[
\left(\frac{33 r}{\sqrt{d}}\right)^d \gamma_1 \log^{d/2}\left(\frac{\gamma_2}{\gamma_1}\right) = \gamma'_d < \lVert f^{=d}\rVert_2 \leq \left(\frac{33\sqrt{d}}{d \log^{1/2}\left(\frac{\gamma_2}{\gamma_1}\right)}\right)^{d} \gamma_1 \log^{d}\left(\frac{\gamma_2}{\gamma_1}\right),
\]
a contradiction. This completes the proof.
\end{proof}

\subsection{Proof of Theorem \ref{thm:level d intro}}

We prove the following concrete version of Theorem \ref{thm:level d intro}. 

\begin{thm}\label{thm:level-d-concrete}
Let $f \colon  \Omega^n \rightarrow \{0, 1\}$. Let $r > 1$, let $d\le  \frac{1}{4}\log(1/\mathbb{E}[f])$, and suppose that 
        \[\mathbb{E}[f_{S\to x}] \le r^{|S|}  \mathbb{E}[f],\] for all sets $S$ of size $\le d$ and all $x\in \Omega^S.$ Then 
        \[ \lVert  f^{=d}\rVert_2^2 \leq \mathbb{E}^2[f] \left(\frac{2200 r^2 \log(1/\mathbb{E}[f])}{d}\right)^d.
        \]
\end{thm}
\begin{proof}
    Denote $\gamma=\EE[f]$. Since the range of $f$ is $\{0,1\}$, Lemma \ref{lem:restriction global implies global} implies that $f$ is both $(2r,\gamma,d)$-$L_1$-global and $(2\sqrt{r},\sqrt{\gamma},d)$-$L_2$ global. By Theorem~\ref{thm: level d is global} we obtain that the function $f^{=d}$ is $(r',\gamma')$-global, for 
    \begin{equation}\label{Eq:r'gamma'}
    r' = \frac{\sqrt{d}}{\log^{1/2}(1/\sqrt{\gamma})} \qquad \mbox{and} \qquad \gamma'=  \left(\frac{33 \cdot (2r)}{\sqrt{d}}\right)^d \gamma \log^{d/2}\left(\frac{1}{\sqrt{\gamma}}\right).
    \end{equation}
    The theorem now follows from the fact that $\lVert f^{=d}\rVert_2 \le \gamma'.$
\end{proof}

More generally, we have the following. 
\begin{thm}\label{thm:level d for restriction global}
    \label{thm:non bool lvl_d}
        Let $f \colon  \Omega^n \to  \mathbb{R}$.
        Let $\gamma_2>\gamma_1 > 0$, $r > 1,$ and $ d\le  \frac{1}{2}\log(\gamma_2/\gamma_1)$, and suppose that for all sets $S \subset [n]$ with $|S|\leq d$ and for all $x \in \Omega^S$, we have
        \[\lVert f_{S\to x}\rVert_1 \le r^{|S|}  \gamma_1 \qquad \mbox{and} \qquad \lVert f_{S\to x}\rVert_2\le r^{|S|}\gamma_2.
        \]
        Then
        \[ \lVert  f^{=d} \rVert_2^2 \leq \gamma_1^2 \left(\frac{2200 r^2 \log(\gamma_2 /\gamma_1)}{d}\right)^d .
        \]
    \end{thm}
\begin{proof}
    By Lemma \ref{lem:restriction global implies global}, $f$ is both $(2r,\gamma_1,d)$-$L_1$-global and $(2r,\gamma_2,d)$-$L_2$ global. The assertion now follows from Theorem \ref{thm: level d is global}, just like in the proof of Theorem~\ref{thm:level-d-concrete}.
\end{proof}

\subsection{Bounding the $q$-norm of $f^{=d}$}

We conclude this section with a bound on the $q$-norm of $f^{=d}$, which we will use in Sections~\ref{sec:globalization_and_cross_intersection},~\ref{sec:smeared_families}. 
\begin{proposition}
    \label{prop:glob_hyp}
        Let $r>1$ and let $f \colon  \Omega^n \rightarrow \bool$. Suppose that $\mathbb{E}[f_{S\to x}]\le r^{|S|}\gamma$ for all sets $S$ of size $\le d$ and for all $x\in \Omega^S.$ Then for any $q \ge  2$ and any integer $d$ we have
        \[ \lVert  f^{=d} \rVert_q \leq \gamma \left( 400 r \sqrt{q} \sqrt{\max \{ \log(1/\gamma), \, q \}} \right)^d. 
        \]
    \end{proposition}
    \begin{proof}
        The assertion clearly holds for $d=0$. Hence, we may assume $d \geq 1$. By Lemma~\ref{lem:restriction global implies global}, $f$ is $(2r,\sqrt{\gamma},d)$-$L_2$-global and $(2r,\gamma,d)$-$L_1$-global. We divide the proof into two cases.

        \medskip \noindent \emph{Case~1: $1 \le d\le \frac{1}{4}\log(1/\gamma)$.} In this case, by Theorem~\ref{thm: level d is global} we obtain that $f^{=d}$ is $(r',\gamma')$-global, for $(r',\gamma')$ defined in~\eqref{Eq:r'gamma'}. Let 
        $\rho = \frac{1}{3 \sqrt{2}} \min\left(\frac{1}{(r')^{\frac{q-2}{q}}q},\frac{1}{\sqrt{q}}\right)$. By Corollary~\ref{cor:hypercontractivity for global functions - alternative}(1),  
        \[
        \rho^{d} \lVert f^{=d}\rVert_q = \lVert T_{\rho} f^{=d}\rVert_q \le \gamma'^{\frac{q-2}{q}}\lVert f^{=d}\rVert_2^{2/q}\le \gamma'. 
        \]
        If $\rho = \frac{1}{3\sqrt{2} \sqrt{q}}$, then by substituting the value of $\gamma'$ we obtain 
        \[
        \lVert f^{=d}\rVert_q \leq \gamma' \rho^{-d} = \gamma \left( 99\cdot2 \, r \sqrt{q} \sqrt{\frac{\log(1 / \gamma)}{d}} \right)^d \leq \gamma \left( 200 r \sqrt{q} \cdot \max \left\{ \sqrt{q}, \sqrt{\log(1 / \gamma)} \right\} \right)^d,
        \]
        and the assertion follows. If $\rho = \frac{1}{3\sqrt{2} (r')^{\frac{q-2}{q}} q}$, then by substituting the values of $\gamma'$ and $r'$ we obtain 
        \[
        \lVert f^{=d}\rVert_q \leq \gamma' \rho^{-d} = \gamma \left( 99\cdot2 \, r \, r' q \, (r')^{-2/q} \sqrt{\frac{\log(1 / \gamma)}{d}} \right)^d = \gamma \left( 99\cdot2\cdot\sqrt{2} \, r \, q \, (r')^{-2/q} \right)^d.
        \]
        Thus, in order to prove the assertion, it is sufficient to show that
        \begin{equation}\label{Eq:Lev-d-q-norm1}
        \sqrt{q} (r')^{-2/q} \leq \sqrt{2} \max \left\{ \sqrt{q}, \sqrt{\log(1 / \gamma)} \right\}.
        \end{equation}
        Note that $\frac{2}{(r')^2}=\frac{\log(1/\gamma)}{d}$, and thus by our assumption on $d$, we have $4 \leq \frac{2}{(r')^2} \leq \log(1/\gamma)$. Now, if $q \ge \log(1/ \gamma) \ge \frac{2}{(r')^2}$, then $(r')^{-2/q} \leq \sqrt{(q/2)^{2/q}} < \sqrt{2}$, and~\eqref{Eq:Lev-d-q-norm1} follows. Hence, we may assume $q< \log(1/\gamma)$. In this case, we may write 
        \[
        \sqrt{q}(r')^{-2/q}=(q \cdot \left(\frac{1}{(r')^2}\right)^{2/q})^{1/2} \leq (q \cdot \left(\frac{\log(1/\gamma)}{2}\right)^{2/q})^{1/2}=(x_0 \cdot a_0^{2/x_0})^{1/2},
        \]
        where $a_0=\frac{\log(1/\gamma)}{2}>1$, and $x_0=q$ satisfies $2<x_0<\log(1/\gamma)$. As for any $a>1$, the function $g_a \colon x \mapsto x a^{2/x}$ is convex for all $x>0$, we have $x_0 (a_0)^{2/x_0} = g_{a_0}(x_0) \leq \max(g_{a_0}(2), g_{a_0}(\log(1/\gamma)))$, and thus, 
        \[
        \sqrt{q}(r')^{-2/q} \leq \max(2 \frac{\log(1/\gamma)}{2},\log(1/\gamma)\left(\frac{\log(1/\gamma)}{2}\right)^{2/\log(1/\gamma)})^{1/2}=\sqrt{\log(1/\gamma)}\sqrt{\max(1,2)}=\sqrt{2}\sqrt{\log(1/\gamma)},
        \]
        where the last inequality holds since $x^{1/x}<2$ for all $x$. Equation~\eqref{Eq:Lev-d-q-norm1} follows.

        \medskip \noindent \emph{Case~2: $d > \frac{1}{4}\log(1/\gamma)$.}
        To handle this case, we may use the fact that
        \[
        \lVert D_{S,x}[f^{=d}]\rVert_2 = \lVert  (D_{S,x}[f])^{= d-|S|} \rVert_2 \le \lVert D_{S,x}[f]\rVert_2,
        \]
        to obtain that $f^{=d}$ is $(2r,\sqrt{\gamma})$-$L_2$-global, since $f$ is $(2r,\sqrt{\gamma},d)$-$L_2$-global. 
        By Corollary~\ref{cor:hypercontractivity for global functions - alternative}(1), setting $\rho = \frac{1}{6 \sqrt{2} rq}$, we have 
        \[
        \rho^{d}\lVert f^{=d}\rVert_q = \lVert T_{\rho}f^{=d}\rVert_q\le \gamma^{\frac{q-2}{2q}} \lVert f^{=d}\rVert_2^{2/q}\le \sqrt{\gamma}.   
        \]
        By the assumption on $d$, this implies
        \[
        \lVert f^{=d}\rVert_q \le \sqrt{\gamma}(6 \sqrt{2}rq)^d\le \gamma (6 \sqrt{2} e^2 r q)^{d} \le \gamma \left( 64 r \sqrt{q} \sqrt{\max \{ \log(1/\gamma), \, q \}} \right)^d, 
        \]
        as asserted.
    \end{proof}

\section{Quantitative Bounds on the Size of Intersecting Families}
\label{sec:intersecting}

In this section we consider families of subsets of $[n]$, with respect to the biased measure $\mu_p$. A family $\mathcal{F} \subset \mathcal{P}([n])$ is called \emph{intersecting} if for any $A,B \in \mathcal{F}$, we have $A \cap B \neq \emptyset$. Two families $\mathcal{F},\mathcal{G} \subset \mathcal{P}([n])$ are called \emph{cross-intersecting} if $\forall S \in \mathcal{F}, T \in \mathcal{G}$, $S \cap T \neq \emptyset$. 

We identify subsets of $[n]$ with elements of $\{0,1\}^n$, and families $\mathcal{F}$ of subsets of $[n]$ with Boolean-valued functions $f=1_{\mathcal{F}} \colon (\{0,1\}^n,\mu_p) \to \{0,1\}$. Subsequently, we say that Boolean-valued functions are (cross-)intersecting if the corresponding families of subsets of $[n]$ are (cross-)intersecting. We use the notation $\mu_p(f)$ for $\mu_p(\mathcal{F})=\mathbb{E}_{\mu_p}(f)$, and say that $f$ is `large' (resp., `small') if $\mu_p(f)$ is `large' (resp., `small').

Throughout this section, when considering a function $f \colon (\{0,1\}^n,\mu_p) \to \{0,1\}$, we use the notations
\[
\alpha=\mu_p(f), \qquad \sigma^2 = \sum_{i=1}^n \hat f(\{i\})^2, \qquad \delta=\max_i |\hat f(\{i\})|, \qquad m=m(f)=|\{i \colon \hat f(\{i\})^2 \geq \frac{\delta^2}{2}\}|,
\]
where the Fourier coefficients are w.r.t.~the measure $\mu_p$. We say that $f$ has \emph{smeared level-$1$ coefficients} or that $f$ is a \textit{smeared level-$1$ function} if $m(f) \geq 1/p^2$.

\medskip \noindent The first result we prove in this section is the following concrete version of Theorem~\ref{theo:upper_bound_smeared_and_intersect-intro}.
\begin{thm}
\label{theo:upper_bound_smeared_and_intersect}
    Let $f \colon (\{0,1\}^n,\mu_p) \to \{0,1\}$ be intersecting, and assume that $1/\sqrt{m(f)} \leq p \leq 1/2$. Then
    \[
    \mu_p(f) \leq 32 \exp \left( -0.0001/p \right).
    \]
\end{thm}
For any $p>1/\sqrt{n}$, the assertion of the theorem is tight, up to a factor of $O(\log(1/p)\log n)$ in the exponent. Indeed, let $T$ be the \emph{tribes} family with tribes of size $r \approx \frac{1}{p} \log n$ (formally, we write $[n]=T_1 \cup T_2 \cup \ldots \cup T_{n/r}$, where the $T_i$'s are disjoint sets of size $r$, and we set $T=\{S \subset [n] \colon \exists i, T_i \subset S\}$). Let $T'$ be the dual family (defined by $S \in T'$ if and only if $[n] \setminus S \not \in T$). The function $f=1_{T \cap T'}$ is intersecting, satisfies $m(f)=n \geq 1/p^2$, and computation shows that $\mu_p(f)=\exp(-O(\frac{1}{p} \log \frac{1}{p} \log n))$. Another example is the `longest-run-of-ones' function, discussed in~\cite{EKN17,MO12}, whose measure is of the same order of magnitude, for all values of $p>1/\sqrt{n}$ simultaneously.

Recall that Kupavskii and Zakharov~\cite{kupavskii2022spread} showed that for any $k<cn/\log n$, the maximum size of a regular intersecting family of $k$-element subsets of $[n]$ is $\exp(-\Omega \left( \frac{n}{k} \right) ) \binom{n}{k}$. As each such family corresponds to a function on $\{0,1\}^n$ all whose $1$-level Fourier coefficients are equal, we can apply to it Theorem~\ref{theo:upper_bound_smeared_and_intersect}. This yields a bound with the same order of magnitude in the exponent as in Kupavskii and Zakharov's result, up to a polynomial factor in $n, p$ which can be absorbed in the exponent. Hence, Theorem~\ref{theo:upper_bound_smeared_and_intersect} implies the result of~\cite{kupavskii2022spread} for all $k>\sqrt{n}$, up to the value of the constants. It should be mentioned though that the constants in~\cite{kupavskii2022spread} are better than the constants we obtain. 

\medskip The proof of the theorem consists of three main steps which span Sections~\ref{sec:globalization_and_cross_intersection}--\ref{sec:intersecting-proof}. First, we show in Section~\ref{sec:globalization_and_cross_intersection} that a function which cross-intersects a `large' global function must be `small'. Then, we show in Section~\ref{sec:smeared_families} that a function with smeared level-$1$ coefficients that admits a measure-reducing restriction by fixing to zero a small number of coordinates, must be `small' as well. At the third step, presented in Section~\ref{sec:intersecting-proof}, we consider a function $f$ that satisfies the hypothesis of the theorem and assume on the contrary that it is `large'. We show that there exists a global restriction $h=f_{S \to x}$ such that $\mu_p(h) \geq e^{|S|}\mu_p(f)$ and $|S|$ is small. As the function $g=f_{S \to 0}$ cross-intersects the large global function $h$, the result of Section~\ref{sec:globalization_and_cross_intersection} implies that $g$ is `small'. However, this means that $f$ admits a measure-reducing restriction by fixing to zero a small number of coordinates, which contradicts the result of Section~\ref{sec:smeared_families}.   

\medskip The second result we  prove in this section is the following concrete version of Theorem~\ref{theo:vector_intersecting_upper_bound_intro}.
\begin{thm}
\label{theo:vector_intersecting_upper_bound}
    Let $n, k$ be natural numbers such that $2 \log n \leq k \leq \sqrt{n} \log n$. Let $\mathcal{A} \subseteq [k]^n$ be a transitive-symmetric vector-intersecting family. Then
    \[
        \frac{|\mathcal{A}|}{k^n} \leq 128 \exp \left( -\frac{0.0001k}{\log n} \right) .
    \]
    For $k > \sqrt{n} \log n$, we have $\frac{|\mathcal{A}|}{k^n} \leq 128 \exp \left( -0.0001\sqrt{n} \right)$.
\end{thm}
For any $2 \log n \leq k \leq \sqrt{n} \log n$, the assertion of the theorem is tight, up to a factor of $O(\log k\log^2 n)$ in the exponent. To see this, one may take the aforementioned family $T \cap T' \subset \{0,1\}^n$, and define a family $\mathcal{A} \subset [k]^n$ by setting $x \in A$ if and only if there exists $y \in T \cap T'$ such that $\forall i \colon (y_i=1) \Rightarrow (x_i=1)$. The family $\mathcal{A}$ is vector-intersecting and transitive symmetric, and computation shows that $\frac{|\mathcal{A}|}{k^n}=\exp(-O(k\log k \log n))$. This construction follows a strategy suggested in~\cite[Section~4]{EKNS19}. 

The proof of the theorem, presented in Section~\ref{sec:vector_intersecting_families}, relies on an embedding of $\mathcal{A}$ into $((\{0,1\}^k)^n,\mu_p)$, for $p=\frac{\log n}{k}$, and on application of Theorem~\ref{theo:upper_bound_smeared_and_intersect}.

\subsection{A family that cross-intersects a large global family, is small}
\label{sec:globalization_and_cross_intersection}

In this section we show that a Boolean-valued function which cross-intersects a `large' global Boolean-valued function, must be `small'.
\begin{proposition}
    \label{prop:mes_cross_inter}
        Let $p \leq 1/2$, and let $g,h \in L_2(\bool^n, \mu_{p})$ be cross-intersecting Boolean-valued functions. Assume that for some $r>0$, we have $\mu_{p}(h_{S\to x}) \le r^{|S|}\mu_{p}(h)$ for all sets $S \subset [n]$  and for all $x\in \{0,1\}^S.$ Denote $c = \frac{1}{3200r}$. 
        If
        $\mu_{p}(h) > e^{-c/p}$, then $\mu_{p}(g) < 8 e^{-c/p}$.
    \end{proposition}
In the proof of the proposition, we use a lemma which follows a strategy of Friedgut~\cite{F08}.
\begin{lem}
    \label{lem:cross_inner_prod}
        Let $p \leq 1/2$, and let $g,h \in L_2(\bool^n, \mu_{p})$ be cross-intersecting Boolean-valued functions. Then 
        \[ \mu_{p}(h) \mu_{p}(g) \leq \sum_{d=1}^n \left( \frac{p}{1-p} \right)^d | \langle h^{=d}, g \rangle |, \]
        where the inner product is taken with respect to the $\mu_p$ measure.
    \end{lem}

    \begin{proof}[Proof of Lemma~\ref{lem:cross_inner_prod}]
        Let $A^{(1)}=\begin{pmatrix}
            \frac{1-2p}{1-p} & \frac{p}{1-p} \\
            1 & 0
        \end{pmatrix}$, and let $A^{(n)}$ be the $n$-fold tensor product of $A^{(1)}$ with itself. Let $A \colon L_2(\bool^n, \mu_p) \to L_2(\bool^n,\mu_p)$ be the linear operator represented by the matrix $A^{(n)}$. In~\cite[Lemma~2.2]{F08}, it is shown that
        \begin{enumerate}
            \item[(1)] $A^{(n)}$ is a pseudo-disjointness matrix for $\{0,1\}^n$, meaning that if $S \cap T \neq \emptyset$ then $A^{(n)}_{S,T}=0$;

            \item[(2)] The eigenfunctions of $A$ are the characters $\{\chi_S\}_{S \subset [n]}$, and the eigenvalue that corresponds to $\chi_S$ is $\left( -\frac{p}{1-p} \right)^{|S|}$.
        \end{enumerate}
        By~(1), since $g,h$ are cross-intersecting we have
        \[
        \langle Ah, g \rangle_{\mu_p} = 0.
        \]
        By~(2), we have 
        \[
        Ah = A(\sum_{S \subset [n]} \hat h(S)\chi_S) = \sum_S \left( -\frac{p}{1-p} \right)^{|S|} \hat h(S)\chi_S = \sum_{d=0}^n \left( -\frac{p}{1-p} \right)^d h^{=d}.
        \]
        Hence, 
        \[
        0 = \langle Ah, g \rangle_{\mu_p} = \sum_{d=0}^n \left( -\frac{p}{1-p} \right)^d \langle h^{=d}, g \rangle \geq \mu_{p}(h) \mu_{p}(g) - \sum_{d=1}^n \left( \frac{p}{1-p} \right)^d | \langle h^{=d}, g \rangle | .
        \]
        Thus, $\mu_{p}(h) \mu_{p}(g) \leq \sum_{d=1}^n \left( \frac{p}{1-p} \right)^d | \langle h^{=d}, g \rangle |$, as asserted.  
    \end{proof}
    
Now we are ready to present the proof of Proposition~\ref{prop:mes_cross_inter}.
\begin{proof}
        Let $q = \max\{\log(1 / \mu_{p}(g)), 2\}$, let $q'$ be the H\"{o}lder conjugate of $q$, and let $A = \max \{ \log(1 / \mu_{p}(h)), q\}$. 
        By H\"{o}lder's inequality, we have
        \[
            | \langle h^{= d}, g \rangle | \leq 
            \lVert  h^{= d} \rVert_q \lVert  g \rVert_{q'} = 
             \lVert  h^{= d} \rVert_q \cdot \mu_{p}(g) / \mu_{p}(g)^{1/q}
        \]
        where the equality holds since $g$ is Boolean-valued. By the definition of $q$,
        \[ \mu_{p}(g)^{1/q} \geq \mu_{p}(g)^{1/\log(1 / \mu_{p}(g))} = e^{-1}.\]
        By Proposition~\ref{prop:glob_hyp} (since $q \geq 2$),
        \[ \lVert  h^{=d} \rVert_q \leq \mu_{p}(h) \Big( 400 r \sqrt{q} \sqrt{A} \Big)^d .\]
        Hence,
        \begin{equation}
        \label{eq:lvl_d_inner_prod}
            | \langle h^{= d}, g \rangle | \leq 
            e \mu_{p}(h) \mu_{p}(g) \left( 400 r \sqrt{q} \sqrt{A} \right)^{d} .
        \end{equation}
        Substituting~\eqref{eq:lvl_d_inner_prod} into Lemma~\ref{lem:cross_inner_prod}, we obtain
        \begin{equation}\label{eq:lvl_d_inner_prod2}
            \mu_{p}(h) \mu_{p}(g) \leq e \mu_{p}(h) \mu_{p}(g) \sum_{d=1}^n \left( \frac{p}{1-p} 400 r \sqrt{q \cdot A} \right)^d .
        \end{equation}
        If $\frac{p}{1-p} 400 r \sqrt{q \cdot A} \leq 1/4$ then~\eqref{eq:lvl_d_inner_prod2} yields
        \[ 1 \leq e \sum_{d=1}^n (1/4)^d < \frac{e}{3} < 1, \]
        a contradiction. Hence, we have $\frac{p}{1-p} 400 r \sqrt{q \cdot A} > 1/4,$ and consequently,
        \begin{equation}
    \label{eq:cross_main_ineq}
            A \geq \sqrt{q \cdot A} \geq \frac{1}{1600 r \frac{1-p}{p}} \geq \frac{1}{3200rp},
        \end{equation} 
        where the last inequality holds since $p \leq 1/2$. Denoting $c = \frac{1}{3200r}$, we obtain $A > c/p$. By the assumption of the proposition, $\log(1 / \mu_p(h)) < c/p$, and hence, $A = q$. Furthermore, we may assume $c/p > 2$, as otherwise, $8 e^{-c/p} > 1$ and the proposition holds trivially. Thus, $q = A > c/p > 2$, and consequently, $q = \log(1 / \mu_p(g))$. Therefore,~      \eqref{eq:cross_main_ineq} yields
        \[
        \log(1 / \mu_p(g)) = q > c/p,
        \]
        or equivalently, $\mu_p(g)< e^{-c/p}$, as asserted.                
    \end{proof}

\subsection{A smeared level-$1$ family admitting a measure-reducing restriction, is small}
\label{sec:smeared_families}


In this section we show that a smeared level-$1$ function $f \subset (\{0,1\}^n,\mu_p)$ that admits a measure-reducing restriction by fixing to zero a small number of coordinates, must be small. 
\begin{proposition}\label{prop:density-decrease}
Let $f \colon (\cube, \mu_p) \rightarrow \{0, 1\}$ be a Boolean-valued function with $m = m(f) > 1 / p^2$. If for some $S \subset [n]$ with $|S| \leq \frac{1}{4p}$, we have $\frac{\mu_p(f_{S \rightarrow 0})}{\mu_p(f)} < \frac{1}{4}$, then $\mu_p(f) < \exp(-0.001 \sqrt{m})$.
\end{proposition}

In the proof of the proposition, we use a `level-1 inequality for smeared functions', which may be useful on its own. 
\begin{proposition}[Level-$1$ inequality for smeared level-$1$ functions]
    \label{lem:spr_sy_lvl_1}
        Let $f \colon (\cube, \mu_p) \rightarrow \{0, 1\}$ be a Boolean-valued function, and assume that $m(f) > 1 / p^2$. If $\alpha = \mathbb{E}[f] > e^{-\sqrt{m(f)}}$, then
        \[ \sum_{i=1}^n \hat f(\{i\})^2 \leq 750 \, \alpha^2 \log(1 / \alpha).\]
\end{proposition}

\begin{proof}
    Consider the function $f^{=1}=\sum_{i=1}^n \hat f(\{i\})\chi_{\{i\}}$. We have $\lVert  f^{=1} \rVert_2 = \sigma$, $D_{S, x} [f^{=1}] = 0$ for any $S \subset [n]$ with $|S| > 1$, and for each $1 \leq i \leq n$,
        \[
        \lVert  D_{i,x} [f^{=1}] \rVert_2 = | \mathbb{E}[D_{i,x}f] | = |\hat{f}(\{i\})| \, |\chi_i(x_i)| \leq \delta \sqrt{\frac{1-p}{p}} \leq \frac{\sqrt{2} \sigma}{\sqrt{m p}},
        \]
    where the last inequality holds since $m\frac{\delta^2}{2}\leq \sigma^2$ by the definition of $m(f)$. Hence, $f^{=1}$ is $\left( \sqrt{\frac{2}{mp}}, \sigma \right)$-$L_2$-global.

    Note that we may assume $\sqrt{\frac{2}{mp}} \leq \sqrt{\frac{2}{\log(1 / \alpha)}}$, as otherwise, $\log(1 / \alpha) > m p \geq \sqrt{m}$, contradicting the assumption $\alpha>e^{-\sqrt{m}}$. Thus, $f^{=1}$ is $\left( \sqrt{\frac{2}{\log(1/\alpha)}}, \sigma \right)$-$L_2$-global.
    Furthermore, we may assume $1 < \frac{1}{4} \log(1 / \alpha)$, as otherwise $\alpha > e^{-4}$, and hence, $750 \, \alpha^2 \log(1 / \alpha) > 1 \ge \sigma^2$, proving the assertion.
    
    Therefore, we may apply Lemma~\ref{lem:level d given globalness} to the function $f$, with $d=1$, $\gamma_1 = \lVert  f \rVert_1 = \alpha$, $\gamma_2 = \lVert f\rVert_2 = \sqrt{\alpha}$, and $r = \sqrt{\frac{2}{\log(1 / \alpha)}}$, to obtain
    \[
    \sigma^2 \leq \frac{33}{\sqrt{2}} \alpha \sigma \sqrt{\log(1 / \alpha)}, \]
    and consequently, $\sigma^2 \leq 750 \alpha^2 \log(1/\alpha)$, as asserted.
    \end{proof}

Now we are ready to present the proof of Proposition~\ref{prop:density-decrease}.

\begin{proof}
The proof of the proposition consists of three steps. We consider a function $f$ which satisfies the hypothesis and assume on the contrary that $f$ is `large' (formally, that $\mu_p(f) \geq \exp(-0.001 \sqrt{m})$). First, we show that $\mathbb{E}[f(x)|x_S \neq 0]$ is large. Then, we show that there exists $i \in S$ such that $f_{i \to 1}$ is large, and deduce that $|\hat f(\{i\})|$ is large. Finally, we use the smearness of $f$ to infer that $\sum_i \hat f(\{i\})^2$ is large, which yields a contradiction by Proposition~\ref{lem:spr_sy_lvl_1}. 

\medskip \noindent \emph{Step~1: $\mathbb{E}[f(x)|x_S \neq 0]$ is large.}
        Denote $|S|=s$ and $a = (1-p)^s$. We claim that
        \begin{equation}\label{Eq:Measure-decrease1}
        \mathbb{E}[f(x)|x_S \neq 0] \geq \frac{1 - a/4}{1 - a} \cdot \mu_p(f).
        \end{equation}
        Indeed, by the assumption on $\mu_p(f_{S \to 0})$, we have 
        \[ \mu_p(f) = \E[f(x) \mid x_S = 0] \cdot \Pr(x_S = 0) + \E[f(x) \mid x_S \neq 0] \cdot \Pr(x_S \neq 0) \leq a \cdot \frac{\mu_p(f)}{4} + (1-a) \E[f(x) \mid x_S \neq 0].\]
        Equation~\eqref{Eq:Measure-decrease1} follows by rearranging.

\medskip \noindent \emph{Step~2: There exists $i \in S$ such that $\hat f(\{i\})$ is large.} First, we claim that there exists $i \in S$ such that
        \begin{equation}\label{Eq:Measure-decrease2}
            \mu_p(f_{i \to 1}) \geq 3\mu_p(f).
        \end{equation}
        Indeed, by a union bound, we have 
        \begin{align*}
            (1-a) \E[f(x) \mid x_S \neq 0] = \E[f(x) \mid x_S \neq 0] \cdot \Pr(x_S \neq 0) \leq \sum_{i \in S} \mu_p(f_{i \to 1}) \cdot \Pr(x_i = 1) \leq s p \cdot \max_{i \in S} \mu_p(f_{i \to 1}).
        \end{align*}
       By rearranging and substituting into~\eqref{Eq:Measure-decrease1}, we obtain $\mu_p(f_{i \to 1}) \geq \frac{1-a/4}{sp}\mu_p(f)$ for some $i$. Equation~\eqref{Eq:Measure-decrease2} follows, since by assumption, $sp \leq 1/4$ and $a \leq 1$. Now, for $i$ that satisfies~\eqref{Eq:Measure-decrease2}, we have 
        \begin{equation}\label{Eq:Measure-decrease3}
            \sqrt{\frac{1-p}{p}} \hat f(\{i\}) = \mu_p(f_{i \to 1}) - \mu_p(f) \geq 2\mu_p(f).
        \end{equation}

\medskip \noindent \emph{Step~3: $\sum_i \hat f(\{i\})^2$ is large, contradicting Proposition~\ref{lem:spr_sy_lvl_1}.} By~\eqref{Eq:Measure-decrease3}, we have $\delta = \max_{i} |\hat f(\{i\})| \geq \sqrt{\frac{p}{1-p}} \cdot 2\mu_p(f)$. By the definition of $m(f)$, this implies 
        \[
        \sum_{i=1}^n \hat f(\{i\})^2 \geq m \frac{p}{1-p} \alpha^2 .
        \]
        On the other hand, by Proposition~\ref{lem:spr_sy_lvl_1}, we have
        \[
        \sum_{i=1}^n \hat f(\{i\})^2 \leq 750 \, \alpha^2 \log(1 / \alpha).
        \]
        Combining these two inequalities, we obtain $m \frac{p}{1-p} \leq 750\log(1 / \alpha)$, which implies 
        \[
        \log(1/\alpha) \geq \frac{1}{750}mp > 0.001 \sqrt{m},
        \]
        contradicting the assumption $\mu_p(f) \geq \exp(-0.001 \sqrt{m})$. This completes the proof.
\end{proof}

\subsection{Proof of Theorem~\ref{theo:upper_bound_smeared_and_intersect}}
\label{sec:intersecting-proof}

Let us recall the statement of the theorem.

\medskip \noindent \textbf{Theorem~\ref{theo:upper_bound_smeared_and_intersect}.} Let $f \colon (\{0,1\}^n,\mu_p) \to \{0,1\}$ be intersecting, and assume that $1/\sqrt{m(f)} \leq p \leq 1/2$. Then $\mu_p(f) \leq 32 \exp \left( -0.0001/p\right)$.

\begin{proof}
    Let $f$ be a function that satisfies the hypothesis, and assume on the contrary $\mu_p(f) > 32 \exp \left( -0.0001/p \right)$. Since $1 / \sqrt{m} \leq p$, this implies   
    \begin{equation}\label{Eq:Proof-intersecting1}
        \mu_p(f) > \exp(-0.001 \sqrt{m}) .        
    \end{equation}
    \emph{Step~1: Constructing a global restriction of $f$.} Fix some `globality parameter' $r > 1$, say $r = e$. We want to find a subset $S \subseteq [n]$ and $x \in \bool^S$ such that $h=f_{S\to x}$ is `global' -- namely, such that for any $T \subseteq [n] \setminus S$ and for any $y \in \bool^T$, we have $\mu_p(h_{T \to y}) \leq \mu_p(h) r^{|T|}$. 
    For this sake, we choose $S, x$ such that $\frac{\mu_p(f_{S \rightarrow x})}{r^{|S|}}$ is maximal (over all choices of $S,x$), and set $h=f_{S \to x}$. By the choice of $S,x$, for any $T \subseteq [n] \setminus S$ and for any $y \in \bool^T$, we have
        \[ \frac{\mu_p(h_{T \rightarrow y})}{r^{|S| + |T|}} = \frac{\mu_p(f_{S \cup T \, \rightarrow (x, y)})}{r^{|S| + |T|}} \leq \frac{\mu_p(f_{S \rightarrow x})}{r^{|S|}} = \frac{\mu_p(h)}{r^{|S|}},\]
    and thus, $\mu_p(h_{T \rightarrow y}) \leq \mu_p(h) \cdot r^{|T|},$ as desired. In addition, we have 
        \[
            \exp \left(-\frac{0.0001}{p} \right) \leq \mu_p(f) = \frac{\mu_p(f)}{r^{0}} \leq \frac{\mu_p(f_{S \rightarrow x})}{r^{|S|}} \leq e^{-|S|},
        \]
        and consequently,
        \begin{equation}\label{Eq:Proof-intersecting2}
            |S| \leq \frac{0.0001}{p}.
        \end{equation}
     \emph{Step~2: Finding a function which cross-intersects the global function $h$ and applying to it Proposition~\ref{prop:mes_cross_inter}.}
        Define $g = f_{S \rightarrow 0}$. The function $g$ cross-intersects $h$, and we have $\mu_p(h_{T \rightarrow y}) \leq \mu_p(h) \cdot r^{|T|}$ for any $T, y$ and $\mu_p(h) \geq \mu_p(f)>\exp \left( -0.0001/p \right)$. Hence, we may apply Proposition~\ref{prop:mes_cross_inter} to deduce that 
        \begin{equation}\label{Eq:Proof-intersecting3}
        \mu_p(g) = \mu_p(f_{S \to 0}) \leq 8 \exp \left( -\frac{0.0001}{p} \right) \leq \frac{1}{4}\mu_p(f).
        \end{equation}
    \emph{Step~3: Obtaining a contradiction using Proposition~\ref{prop:density-decrease}}. By Equations~\eqref{Eq:Proof-intersecting2} and~\eqref{Eq:Proof-intersecting3}, the smeared level-$1$ function $f$ admits a restriction $g=f_{S \to 0}$ such that $\frac{\mu_p(g)}{\mu_p(f)}\leq \frac{1}{4}$, for $|S|<\frac{1}{4p}$. By Proposition~\ref{prop:density-decrease}, this implies that $\mu_p(f)<\exp(-0.001 \sqrt{m})$, which contradicts~\eqref{Eq:Proof-intersecting1}. This completes the proof.
\end{proof}

\subsection{Intersecting transitive-symmetric families of vectors}
\label{sec:vector_intersecting_families}

In this section we prove Theorem~\ref{theo:vector_intersecting_upper_bound}. Let us recall the formulation of the theorem. 

\medskip \noindent \textbf{Theorem~\ref{theo:vector_intersecting_upper_bound}.}
    Let $n, k \in \mathbb{N}$ such that $2 \log n \leq k \leq \sqrt{n} \log n$. Let $\mathcal{A} \subseteq [k]^n$ be a transitive-symmetric vector-intersecting family. Then
    \[
        \frac{|\mathcal{A}|}{k^n} \leq 128 \exp \left( -\frac{0.0001k}{\log n} \right) .
    \]
    For $k > \sqrt{n} \log n$, we have $\frac{|\mathcal{A}|}{k^n} \leq 128 \exp \left( -0.0001\sqrt{n} \right)$.

\medskip \noindent To prove the theorem, we show that $\mathcal{A}$ can be embedded into $((\{0,1\}^k)^n,\mu_p)$, for $p=\frac{\log n}{k}$, in such a way that the resulting family $\mathcal{B}$ is intersecting, has smeared level-$1$ coefficients, and satisfies $\mu_p(\mathcal{B}) \geq \frac{1}{4}\frac{|\mathcal{A}|}{k^n}$. The assertion will then follow from  application of Theorem~\ref{theo:upper_bound_smeared_and_intersect} to $\mathcal{B}$. 

Let us present the embedding. Denote
\[
\slice \defeq \{S \in (\bool^k)^n \mid \forall i \colon |S_i| = 1\}, \qquad \mbox{and} \qquad \geslice \defeq \{S \in (\bool^k)^n \mid \forall i \colon |S_i| \geq 1\}. 
\]
First, we embed $\mathcal{A}$ into $\slice$, to obtain
\[ 
    \tilde{\mathcal{A}} \defeq \{S=(S_1,S_2,\ldots,S_n) \in (\bool^k)^n \mid \exists z \in \mathcal{A} \ s.t. \ \forall i\colon S_i = \{ z_i \} \} \subseteq \slice .
\]
Then we take $\mathcal{B} = \tilde{\mathcal{A}}^\uparrow$ to be the up-closure of $\tilde{\mathcal{A}}$ in $(\{0,1\}^k)^n$ (i.e., we have $x \in \mathcal{B}$ if and only if $\exists y \in \tilde{\mathcal{A}} \colon \forall i, y_i \leq x_i$). The following lemma shows that for an appropriate choice of $p$, $\mu_p(\mathcal{B})$ is not much smaller than $\frac{|\mathcal{A}|}{k^n}$. 
\begin{lem}
\label{lem:compare_mu_p_to_slice}
For any $k, n \in \mathbb{N}, p \in (0, 1]$, and any $\mathcal{A} \subseteq [k]^n$, we have
\[ \frac{|\mathcal{A}|}{k^n} \leq \frac {\mu_{p}(\mathcal{B})} {\left( 1-\left( 1-p \right)^k \right)^n} .\]
\end{lem}
By fixing $p = \frac{\log n}{k}$ and using the monotonicity of the function $x \mapsto (1 - 1/x)^x$ for $x \geq 1$, Lemma~\ref{lem:compare_mu_p_to_slice} implies:
\begin{cor}
\label{cor: compare mu_p to slice}
Let $2 \leq k,n\in \mathbb{N}$. Set $p=\frac{\log n}{k}$. Then for any $\mathcal{A} \subseteq [k]^n$, we have 
\[ \frac{|\mathcal{A}|}{k^n} \leq 4 \mu_{p}(\mathcal{B}) .\]
\end{cor}

\begin{proof}[Proof of Lemma~\ref{lem:compare_mu_p_to_slice}]
The proof uses a coupling argument. Define a distribution $\mu$ on $\geslice$ by conditioning $\mu_p$: For all $x \in \geslice$,
\[ \mu(x) = \frac {\mu_{p}(x)} {\mu_{p}(\geslice)} = \frac {\mu_{p}(x)} {\left( 1-\left( 1-p \right)^k \right)^n}.\]
Define a distribution $\nu$ on $\slice$ by sampling an element $x$ from the distribution $\mu$ and sampling uniformly at random an element of $\slice$ among the elements beneath it: $\{ y \in \slice \mid y \leq x \}$. It is clear that $\nu$ is the uniform distribution on $\slice$. Hence, the marginals of the distribution $(\mu, \nu)$ on $\geslice \times \slice$ satisfy
\[ 
\mu(\mathcal{B}) = \frac {\mu_{p}(\mathcal{B})} {\left( 1-\left( 1-p \right)^k \right)^n}, \qquad \mbox{and} \qquad \nu(\tilde{\mathcal{A}}) = \frac {|\tilde{\mathcal{A}}|} {|\slice|} = \frac{|\mathcal{A}|}{k^n}.
\]
Consider the indicator random variables
\[ 
    X(x, y) = 
    \begin{cases}
        1, & x \in \mathcal{B} \\
        0, & otherwise
    \end{cases}
    \qquad \mbox{and} \qquad Y(x, y) = 
    \begin{cases}
        1, & y \in \tilde{\mathcal{A}} \\
        0, & otherwise
    \end{cases}.
\]
Note that $X \geq Y$. Indeed, by the definition of $\nu$, for any $(x,y)$ drawn from the distribution $(\mu,\nu)$ we have $x \geq y$. Thus,
\[
Y(x,y)=1 \quad \Rightarrow \quad y \in \tilde{\mathcal{A}} \quad \Rightarrow \quad x \in \tilde{\mathcal{A}}^\uparrow = \mathcal{B} \quad \Rightarrow \quad X(x,y)=1. 
\]
Therefore, we have
\[
    \frac {\mu_{p}(\mathcal{B})} {\mu_{p}(\geslice)} = \mu(\mathcal{B}) = \E[X] \geq \E[Y] = \nu(\tilde{\mathcal{A}}) = \frac{|\mathcal{A}|}{k^n}, 
\]
as asserted.
\end{proof}

Now we are ready to present the proof of Theorem~\ref{theo:vector_intersecting_upper_bound}.

\begin{proof}
    Let $2 \log n \leq k \leq \sqrt{n} \log n$, and let $\mathcal{A} \subseteq [k]^n$ be a transitive-symmetric vector-intersecting family. Construct $\mathcal{B} \subset ((\{0,1\}^k)^n,\mu_{\log n/k})$ in the way described above. Clearly, $\mathcal{B}$ is intersecting. 
    
    We claim that $\mathcal{B}$ has $n$-smeared level-$1$ coefficients (meaning $m(\mathcal{B}) \geq n$). Indeed, if $\sigma \in S_n$ is a permutation that preserves $\mathcal{A}$, then $\sigma' \in S_{nk}$ defined by $ \sigma'(i + j \cdot k) = i + \sigma(j) \cdot k $ preserves $\mathcal{B}$. Since $\mathcal{A}$ is transitive-symmetric, there exists a group of permutations $G \leq S_{nk}$ preserving $\mathcal{B}$, which is transitive on each of the following sets of coordinates:
    \[
    \{1, 1+k, \ldots, 1+(n-1)k\}, \{2, 2+k, \ldots, 2+(n-1)k\}, \ldots, \{k, 2k, \ldots, nk\}.
    \]
    Consequently, we have $\hat 1_{\mathcal{B}}(\{i\}) = \hat 1_{\mathcal{B}}(\{i+jk\})$ for all $1 \leq i \leq k$ and all $1 \leq j \leq n$, and thus, $1_{\mathcal{B}}$ has $n$-smeared level-$1$ coefficients.

    Hence, the function $1_{\mathcal{B}}$ satisfies the hypothesis of Theorem~\ref{theo:upper_bound_smeared_and_intersect}. Applying the theorem to it, we obtain $\mu_{\log n/k}(\mathcal{B}) \leq 32 \exp(-0.0001k/\log n)$. As $\frac{|\mathcal{A}|}{k^n} \leq 4 \mu_{\log n/k}(\mathcal{B})$ by Corollary~\ref{cor: compare mu_p to slice}, the assertion follows.

    \medskip To prove the assertion for $k > \sqrt{n} \log n$ and any  transitive-symmetric vector-intersecting family $\mathcal{A} \subseteq [k]^n$, construct $\mathcal{B}$ as above and note that as $\mathcal{B}$ is monotone, combination of Corollary~\ref{cor: compare mu_p to slice} with an application of Theorem~\ref{theo:upper_bound_smeared_and_intersect} for $p = \frac{1}{\sqrt{n}}$ yields
    \[
        \frac{|\mathcal{A}|}{k^n} \leq 4 \mu_{\log n / k}(\mathcal{B}) \leq 4 \mu_{1 / \sqrt{n}}(\mathcal{B}) \leq 128\exp(-0.0001/\sqrt{n}),
    \]
    proving the assertion.
\end{proof}

\bibliographystyle{plain}
\bibliography{refs}

\end{document}